\documentclass[10pt,a4paper]{article}

\usepackage[T1]{fontenc}
\usepackage[utf8]{inputenc}
\usepackage{amsmath,amssymb,amsthm,amsfonts,mathtools}

\usepackage{booktabs}
\usepackage{chngcntr}		
\usepackage{microtype}
\counterwithin{table}{section}
\counterwithin{figure}{section}
\usepackage{tikz}
\usetikzlibrary{arrows}

\theoremstyle{plain}
\newtheorem{theorem}{Theorem}[section]
\newtheorem{proposition}[theorem]{Proposition}
\newtheorem{lemma}[theorem]{Lemma}

\theoremstyle{definition}
\newtheorem{definition}[theorem]{Definition}
\newtheorem{remark}[theorem]{Remark}
\newtheorem{example}[theorem]{Example}

\theoremstyle{remark}
\numberwithin{equation}{section}

\usepackage[pdfborder={0 0 0},bookmarksopen]{hyperref}
\hypersetup{%
  pdftitle = {Robust Pricing and Hedging around the Globe},
  pdfauthor = {Sebastian Herrmann, Florian Stebegg},
  pdfpagemode = UseNone,
}

\newcommand{\FF}{\mathbb{F}}
\newcommand{\NN}{\mathbb{N}}
\newcommand{\RR}{\mathbb{R}}

\newcommand{\bfI}{\mathbf{I}}
\newcommand{\bfS}{\mathbf{S}}

\newcommand{\cA}{\mathcal{A}}
\newcommand{\cD}{\mathcal{D}}
\newcommand{\cF}{\mathcal{F}}
\newcommand{\cL}{\mathcal{L}}
\newcommand{\cM}{\mathcal{M}}
\newcommand{\cT}{\mathcal{T}}

\newcommand{\diff}{\mathrm{d}}
\newcommand{\dd}{\,\mathrm{d}}
\newcommand{\bary}{\mathrm{bary}}
\newcommand{\supp}{\mathrm{supp}}

\newcommand{\1}{\mathbf{1}}

\newcommand{\ldbrack}{[\![}
\newcommand{\rdbrack}{]\!]}

\newcommand*{\EX}[2][]{E^{#1}\left [ #2 \right ]}
\newcommand*{\cEX}[3][]{E^{#1}\left[ #2 \;\middle\vert\; #3 \right]}
\newcommand*{\wt}[1]{\widetilde{#1}}
\newcommand*{\ol}[1]{\overline{#1}}

\begin{document}
\title{%
Robust Pricing and Hedging around the Globe%
\footnote{The authors thank David Hobson, Sigrid K\"allblad, Marcel Nutz, and Yavor Stoev for stimulating discussions.}
}
\date{\today}
\author{%
  Sebastian Herrmann%
  \thanks{
  Department of Mathematics, University of Michigan, email
  \href{mailto:sherrma@umich.edu}{\nolinkurl{sherrma@umich.edu}}.
  }
  \and
  Florian Stebegg%
  \thanks{
  Department of Statistics, Columbia University,
  email
  \href{florian.stebegg@columbia.edu}{\nolinkurl{florian.stebegg@columbia.edu}}.
  }
}
\maketitle

\begin{abstract}
We consider the martingale optimal transport duality for c\`adl\`ag processes with given initial and terminal laws. Strong duality and existence of dual optimizers (robust semi-static superhedging strategies) are proved for a class of payoffs that includes American, Asian, Bermudan, and European options with intermediate maturity. We exhibit an optimal superhedging strategy for which the static part solves an auxiliary problem and the dynamic part is given explicitly in terms of the static part.
\end{abstract}

\vspace{0.5em}

{\small
\noindent \emph{Keywords} Robust superhedging; Semi-static strategies; Martingale optimal transport; Duality.

\vspace{0.25em}
\noindent \emph{AMS MSC 2010}
60G44, 
49N05, 
91G20. 

\vspace{0.25em}
\noindent \emph{JEL Classification}
G12, 
G23, 
C61. 
}

\section{Introduction}

This paper studies the robust pricing and superhedging of derivative securities with a payoff of the form
\begin{align}
\label{eqn:intro:payoff}
F(X,A)
&= f\Big(\int_{[0,T]} X_t \dd A_t\Big).
\end{align}
Here, $f$ is a nonnegative Borel function, $X$ is a c\`adl\`ag price process (realized on the Skorokhod space), and $A$ is chosen by the buyer from a given set $\cA$ of exercise rights. More precisely, $\cA$ is a set of so-called averaging processes, i.e., nonnegative and nondecreasing adapted c\`adl\`ag processes $A$ with $A_T \equiv 1$. Setting $\cA = \lbrace \1_{\ldbrack \tau, T \rdbrack}: \tau \text{ a }[0,T]\text{-valued stopping time}\rbrace$ or $\cA = \lbrace t \mapsto t/T \rbrace$ reduces \eqref{eqn:intro:payoff} to the relevant special cases of American- or Asian-style derivatives, respectively:
\begin{align}
\label{eqn:intro:American and Asian}
f(X_\tau)
\quad\text{or}\quad
f\Big(\frac{1}{T}\int_0^T X_t \dd t\Big).
\end{align}
Other relevant examples are Bermudan options and European options with intermediate maturity (cf.~Examples~\ref{ex:asian european}--\ref{ex:american bermudan}).

\paragraph{Robust pricing problem.}
Let $\mu$ and $\nu$ be probability measures on $\RR$. We denote by $\cM(\mu,\nu)$ the set of (continuous-time) martingale couplings between $\mu$ and $\nu$, i.e., probability measures $P$ under which $X$ is a martingale with marginal distributions $X_0\stackrel{P}{\sim}\mu$ and $X_T\stackrel{P}{\sim} \nu$. The value of the \emph{primal problem}
\begin{equation}
\label{eqn:intro:S}
\bfS
:=\sup_{P\in\cM(\mu,\nu)} \sup_{A\in\cA} \EX[P]{F(X,A)}
\end{equation}
can be interpreted as the maximal model-based price for $F$ over all models which are consistent with the given marginals.

If $\cA$ is a singleton, then \eqref{eqn:intro:S} is a so-called (continuous-time) martingale optimal transport problem. This problem was introduced (for general payoffs) by Beiglb\"ock, Henry-Labord\`ere, and Penkner \cite{BeiglbockHenryLaborderePenkner2013} in a discrete-time setting and by Galichon, Henry-Labord\`ere, and Touzi \cite{GalichonHenryLabordereTouzi2014} in continuous time; cf.~the survey~\cite{Touzi2014}.

\paragraph{Semi-static superhedging problem.}
The formal dual problem to \eqref{eqn:intro:S} has a natural interpretation as a superhedging problem.\footnote{The primal problem \eqref{eqn:intro:S} can be viewed as an optimization over finite measures $P$ with three constraints: two marginal constraints and the martingale constraint. Its formal dual problem is the Lagrange dual problem where suitable functions $\varphi$ and $\psi$ and a suitable process $H$ are used as Lagrange multipliers for the marginal and martingale constraints, respectively.} Loosely speaking, a semi-static superhedge is a triplet $(\varphi,\psi,H)$ consisting of functions $\varphi,\psi$ and a suitable process $H$ such that for every $A \in \cA$, the superhedging inequality holds:
\begin{equation}
\label{eqn:intro:superhedging inequality}
\varphi(X_{0}) + \psi(X_{T}) + \int_{0}^{T} H^A_{t-} \dd X_t
\geq F(X,A)\quad\text{pathwise}.
\end{equation}
Here, the strategy $H = H^A$ may depend in an adapted way on $A$ (cf.~Section~3.2 for a precise formulation). For the example of an American-style payoff, this means that at the chosen exercise time $\tau$, the buyer communicates her decision to exercise to the seller, who can then adjust the dynamic part of his hedging strategy (cf.~\cite[Section~3]{BayraktarHuangZhou2015}). The left-hand side in \eqref{eqn:intro:superhedging inequality} is the payoff of a static position in two European-style derivatives on $X$ plus the final value of a self-financing dynamic trading strategy in $X$. The inequality \eqref{eqn:intro:superhedging inequality} says that the final value of this semi-static portfolio dominates the payoff $F$ for every choice of exercise right and ``all'' price paths. The initial cost to set up a semi-static  superhedge $(\varphi,\psi,H)$ equals the price $\mu(\varphi)+\nu(\psi)$ of the static part.\footnote{We use the common notation $\mu(\varphi)$ for the integral of $\varphi$ against $\mu$.} The formal \emph{dual problem} to \eqref{eqn:intro:S},
\begin{equation}
\label{eqn:intro:I}
\bfI
:= \inf \lbrace \mu(\varphi) + \nu(\psi) : (\varphi,\psi,H) \text{ is a semi-static superhedge}\rbrace,
\end{equation}
amounts to finding the cheapest semi-static superhedge (if it exists) and its initial cost, the so-called robust superhedging price.

\paragraph{Main objectives and relaxation of the dual problem.}
We are interested in \emph{strong duality}, i.e., $\bfS = \bfI$, and \emph{dual attainment}, i.e., the existence of a dual minimizer.

Dual attainment requires a suitable relaxation of the dual problem. Indeed, for the discrete-time martingale optimal transport problem, Beiglb\"ock, Henry-Labord\`ere, and Penkner~\cite{BeiglbockHenryLaborderePenkner2013} show strong duality for upper semicontinuous payoffs but provide a counterexample that shows that dual attainment can fail even if the payoff function is bounded and continuous. Beiglb\"ock, Nutz, and Touzi~\cite{BeiglbockNutzTouzi2017} achieve strong duality and dual attainment for general payoffs and marginals in the one-step case by relaxing the dual problem in two aspects. First, they only require the superhedging inequality to hold in the quasi-sure sense, i.e., outside a set which is a nullset under every one-step martingale coupling between $\mu$ and $\nu$. The reason is that the marginal constraints may introduce barriers on the real line which (almost surely) cannot be crossed by \emph{any} martingale with these marginals; this was first observed by Hobson~\cite{Hobson1998.maximum} (see also Cox~\cite{Cox2008} and Beiglb\"ock and Juillet~\cite{BeiglbockJuillet2016}). These barriers partition the real line into intervals and the marginal laws into so-called irreducible components. Then strong duality and dual attainment can be reduced to proving the same results for each irreducible component \cite{Hobson1998.maximum,BeiglbockNutzTouzi2017}. Second, Beiglb\"ock, Nutz, and Touzi~\cite{BeiglbockNutzTouzi2017} extend the meaning of the expression $\mu(\varphi) + \nu(\psi)$ to certain situations where both individual integrals are infinite. For example, it can happen that $\mu(\varphi) = -\infty$ and $\nu(\psi) = \infty$, but the price $\EX[P]{\varphi(X_0) + \psi(X_T)}$ of the combined static part is well-defined, finite, and invariant under the choice of $P\in\cM(\mu,\nu)$. In this situation, this price is still denoted by $\mu(\varphi)+\nu(\psi)$. We employ both relaxations for the precise definition of the dual problem in Section~\ref{sec:superhedging problem}.

In continuous time, Dolinsky and Soner~\cite{DolinskySoner2014,DolinskySoner2015} show strong duality for uniformly continuous payoffs satisfying a certain growth condition. They use the integration by parts formula to define the stochastic integral $\int_0^T H_{t-} \dd X_t$ pathwise for finite variation integrands $H$. However, dual attainment cannot be expected in this class in general. For our payoffs \eqref{eqn:intro:payoff}, we need to allow integrands that are of finite variation whenever they are bounded but can become arbitrarily large or small on certain price paths. As the integrands are not of finite variation on these paths, the meaning of the pathwise integral needs to be extended appropriately.

For the purpose of the introduction, we discuss our results and methodology in a non-rigorous fashion, ignoring all aspects relating to the relaxation of the dual problem.

\paragraph{Main results.}
We prove strong duality and dual attainment for payoffs of the form \eqref{eqn:intro:payoff} under mild conditions on $f$ and $\cA$ for irreducible marginals (Theorem~\ref{thm:StrongDuality}); all results can be extended to general marginals along the lines of \cite[Section~7]{BeiglbockNutzTouzi2017}. The key idea is the reduction of the primal and dual problems to simpler auxiliary problems, which do not depend on the set $\cA$ of exercising rights. In particular, our results cover American-style derivatives $f(X_\tau)$ for Borel-measurable $f$ and Asian-style derivatives $f(\frac{1}{T}\int_0^T X_t \dd t)$ for lower semicontinuous $f$ and show that both derivatives have (perhaps surprisingly) the same robust superhedging prices and structurally similar semi-static superhedges.

\paragraph{Methodology.}
Our methodology relies on two crucial observations which allow us to bound the primal problem from below and the dual problem from above by simpler auxiliary primal and dual problems, respectively. To obtain a primal lower bound, we show that for any law $\theta$ which is in convex order between $\mu$ and $\nu$, there is a sequence $(P_n)_{n\geq1} \subset \cM(\mu,\nu)$ such that the law of $\int_{[0,T]}X_t\dd A_t$ under $P_n$ converges weakly to $\theta$ if $A$ is a suitable averaging process. This allows us to bound $\bfS$ from below by the value of the auxiliary primal problem 
\begin{equation*}
\wt\bfS
:= \sup_{\mu\leq_c\theta\leq_c\nu} \theta(f).
\end{equation*}
(The converse inequality also holds (cf.~Lemma~\ref{lem:average:convex order}), so that in fact $\bfS = \wt\bfS$.)

Regarding the dual upper bound, we prove (modulo technicalities) that if $\varphi$ is concave and $\psi$ is convex such that $\varphi + \psi \geq f$, then $(\varphi,\psi,H)$ is a semi-static superhedge, where the dynamic part $H$ is given explicitly in terms of $\varphi$ and $\psi$ by
\begin{equation}
\label{eqn:intro:dynamic part}
H_t
:= \varphi'(X_0) - \int_{[0,t]}\lbrace \varphi'(X_0) + \psi'(X_s) \rbrace \dd A_s.
\end{equation}
This allows us to bound $\bfI$ from above by the value of the auxiliary dual problem
\begin{align*}
\wt\bfI
&:= \inf \lbrace \mu(\varphi) + \nu(\psi) : \text{$\varphi$ concave, $\psi$ convex, $\varphi + \psi \geq f$} \rbrace.
\end{align*}

As a consequence, strong duality and dual attainment for $\bfS$ and $\bfI$ follow from the same assertions for the simpler problems $\wt\bfS$ and $\wt\bfI$, which we prove by adapting the techniques of \cite{BeiglbockNutzTouzi2017}. Moreover, our reduction of the dual problem implies that if $(\varphi,\psi)$ is optimal for $\wt\bfI$, then it is also the static part of an optimal semi-static superhedge and the dynamic part $H$ can be computed ex-post via \eqref{eqn:intro:dynamic part}. This dramatically decreases the complexity of the superhedging problem: the optimization over two functions \emph{and a process} satisfying an inequality constraint \emph{on the Skorokhod space} is reduced to an optimization over two functions satisfying an inequality constraint on $\RR$.

Our methodology reveals that many derivatives have the same robust superhedging prices and semi-static superhedges. Indeed, $\wt\bfI$ and $\wt\bfS$ do not depend on the set $\cA$ of exercise rights granted to the buyer, and this independence transfers to $\bfS$ and $\bfI$ under mild conditions on $f$ and $\cA$. For example, if $f$ is lower semicontinuous, then the Asian-style derivative $f(\frac{1}{T}\int_0^T X_t \dd t)$, the American-style derivative $f(X_\tau)$, and the European-style derivative $f(X_{T'})$ (for a fixed maturity $T' \in (0,T)$) all have the same robust superhedging price (Remark~\ref{rem:strong duality:derivatives}). This invariance breaks down when more than two marginals are given.

\paragraph{Related literature.}
Much of the extant literature on robust superhedging in semi-static settings is concerned with strong duality and dual attainment. The results vary in their generality and explicitness as well as their precise formulation. The semi-static setting, where call options are available as additional hedging instruments, dates back to Hobson's seminal paper \cite{Hobson1998} on the lookback option.\footnote{We note that the superhedging strategies described in \cite{Hobson1998} are actually dynamic in the call options.} Many other specific exotic derivatives (mostly without special exercise rights) have been analyzed in this framework in the past two decades; cf.~the survey \cite{Hobson2011}.

Securities with special exercise rights have been studied in the context of American-style derivatives in discrete-time settings. Bayraktar, Huang, and Zhou \cite{BayraktarHuangZhou2015} obtain a duality result for a somewhat different primal problem (cf.~\cite[Theorem~3.1]{BayraktarHuangZhou2015}) and show that duality may fail in their setting if they formulate their primal problem in analogy to the present paper; see also \cite{BayraktarZhou2017} for related results with portfolio constraints. Hobson and Neuberger~\cite{HobsonNeuberger2017} (based on Neuberger's earlier manuscript \cite{Neuberger2007}) resolve this issue by adopting a weak formulation for the primal problem: instead of optimizing only over martingale measures on a fixed filtered path space, the optimization there runs over filtered probability spaces supporting a martingale and thereby allows richer information structures and hence more stopping times. We also refer to \cite{AksamitDengOblojTan2018, BayraktarZhou2016.super-hedging, HobsonNeuberger2016} for recent developments in this regard. We note that all these papers permit significant restrictions on the set of possible price paths (e.g., binomial trees) while we allow all c\`adl\`ag paths. This difference may be the reason why strong duality holds in our setting without any relaxation of the primal problem.

The case of an Asian-style payoff as in \eqref{eqn:intro:American and Asian} has been studied in the case of a Dirac initial law $\mu$. For convex or concave $f$, Stebegg~\cite{Stebegg2014} shows strong duality and dual attainment. For nonnegative Lipschitz $f$, Cox and K\"allblad~\cite{CoxKallblad2017} obtain a PDE characterization of the maximal model-based price for finitely supported $\nu$. Bayraktar, Cox, and Stoev~\cite{BayraktarCoxStoev2018} provide a similar, but not identical PDE for the corresponding pricing problem for American-style payoffs as in \eqref{eqn:intro:American and Asian}. A consequence of our main duality result is that the Asian and American pricing problems are the same, so that both these PDEs have the same (viscosity) solution.

\paragraph{Organization of the paper.}
The remainder of the article is organized as follows. In Section~2, we recall basic results on the convex order and potential functions, introduce the generalized integral of \cite{BeiglbockNutzTouzi2017} and its relevant properties, and present the extension of the pathwise definition of the stochastic integral for finite variation integrands. Section~\ref{sec:problems} introduces the robust pricing and semi-static superhedging problems and presents our duality result. The duality between the auxiliary problems, the structure of their optimizers, and their relation to the robust pricing and superhedging problem are treated in Section~\ref{sec:auxiliary problems}. In Section~\ref{sec:examples}, we provide simple geometric constructions of primal and dual optimizers for risk reversals and butterfly spreads. Finally, some counterexamples are collected in Section~\ref{sec:counterexamples}.

\section{Preliminaries}

Fix a time horizon $T$ and let $\Omega = D([0,T];\RR)$ be the space of real-valued c\`adl\`ag paths on $[0,T]$. We endow $\Omega$ with the Skorokhod topology and denote by $\cF$ the corresponding Borel $\sigma$-algebra, by $X = (X_t)_{t\in[0,T]}$ the canonical process on $\Omega$, and by $\FF = (\cF_t)_{t\in[0,T]}$ the (raw) filtration generated by $X$. Unless otherwise stated, all probabilistic notions requiring a filtration pertain to $\FF$.

For any process $Y = (Y_t)_{t\in[0,T]}$ on $\Omega$, we set $Y_{0-} = 0$, so that the jump of $Y$ at time $0$ is $\Delta Y_0 = Y_0$.

\subsection{Martingale measures and convex order}

Let $\mu$ and $\nu$ be finite\footnote{As in \cite{BeiglbockNutzTouzi2017}, using finite measures as opposed to probability measures turns out to be useful.} measures on $\RR$ with finite first moment. We denote by $\Pi(\mu,\nu)$ the set of (continuous-time) couplings of $\mu$ and $\nu$, i.e., finite measures $P$ on $\Omega$ such that $P \circ X_0^{-1} = \mu$ and $P \circ X_T^{-1} = \nu$. If, in addition, the canonical process $X$ is a martingale under $P$ (defined in the natural way if $P$ is not a probability measure), then we write $P\in\cM(\mu,\nu)$ and say that $P$ is a (continuous-time) martingale coupling between $\mu$ and $\nu$.

We also consider discrete-time versions of these notions. To wit, we denote by $\Pi^d(\mu,\nu)$ the set of finite measures $Q$ on $\RR^2$ with marginal distributions $\mu$ and $\nu$ and by $\cM^d(\mu,\nu)$ the subset of measures $Q$ under which the canonical process on $\RR^2$ is a martingale (in its own filtration). The sets $\Pi^d(\mu,\theta,\nu)$ and $\cM(\mu,\theta,\nu)$ of finite measures on $\RR^3$ with prescribed marginal distributions are defined analogously.

We write $\mu \leq_c \nu$ if $\mu$ and $\nu$ are in \emph{convex order} in the sense that $\mu(\varphi) \leq \nu(\varphi)$ holds for any convex function $\varphi : \RR \to \RR$. In this case, $\mu$ and $\nu$ have the same mass and the same barycenter $\bary(\mu) := \frac{1}{\mu(\RR)}\int x \, \mu(\diff x)$.

The potential function $u_\mu:\RR\to[0,\infty]$ of $\mu$ is defined as 
\begin{align}
\label{eqn:potential function}
u_\mu(x)
&:= \int |x-y| \,\mu(\diff y).
\end{align}
We refer to \cite[Section~4.1]{BeiglbockJuillet2016} for basic properties of potential functions. In particular, the following relationship between the convex order, potential functions, and martingale measures is well known.

\begin{proposition}
\label{prop:Strassen}
Let $\mu$ and $\nu$ be finite measures with finite first moments and $\mu(\RR) = \nu(\RR)$. Then the following are equivalent: (i) $\mu \leq_c \nu$, (ii) $u_\mu \leq u_\nu$, (iii) $\cM^d(\mu,\nu) \neq \emptyset$, and (iv) $\cM(\mu,\nu)\neq\emptyset$.
\end{proposition}

An analogous result holds for three marginals $\mu$, $\theta$, $\nu$, the corresponding potential functions, and the set $\cM^d(\mu,\theta,\nu)$.

We recall the following definition from \cite[Definition~2.2]{BeiglbockNutzTouzi2017} (see also \cite[Definition~A.3]{BeiglbockJuillet2016}).

\begin{definition}
\label{def:irreducible}
A pair of finite measures $\mu \leq_c \nu$ is called \emph{irreducible} if the set $I = \{u_\mu < u_\nu\}$ is connected and $\mu(I) = \mu(\RR)$. In this situation, let $J$ be the union of $I$ and any endpoints of $I$ that are atoms of $\nu$; then $(I,J)$ is the \emph{domain} of $\mu \leq_c \nu$.
\end{definition}

We work with irreducible $\mu \leq_c \nu$ for the remainder of this article.

\subsection{Generalized integral}
\label{sec:BNT integral}

Let $\mu \leq_c \nu$ be irreducible with domain $(I,J)$. Beiglb\"ock and Juillet~\cite[Section~A.3]{BeiglbockJuillet2016} and Beiglb\"ock, Nutz, and Touzi \cite[Section~4]{BeiglbockNutzTouzi2017} appropriately extend the meaning of the expression $\mu(\varphi) + \nu(\psi)$ to the case where the individual integrals are not necessarily finite. We present here a slight extension of their work in order to deal with intermediate laws $\mu\leq_c\theta\leq_c\nu$ for which the pairs $\mu\leq_c\theta$ and $\theta\leq_c\nu$ may not be irreducible.

For the rest of this article, whenever we write $\mu \leq_c \nu$ for any two measures $\mu$ and $\nu$, we implicitly assume that both measures are finite and have a finite first moment. Throughout this section, we fix $\mu\leq_c\theta_1\leq_c\theta_2\leq_c\nu$.

\begin{definition}
\label{def:concave integral}
Let $\chi:J\to\RR$ be concave. Denote by $-\chi''$ the second derivative measure of the convex function $-\chi$ and by $\Delta \chi$ the possible jumps of $\chi$ at the endpoints of $I$. We set
\begin{align}
\label{eqn:def:concave integral}
(\theta_1-\theta_2)(\chi)
&:= \frac{1}{2}\int_I (u_{\theta_1}-u_{\theta_2}) \dd \chi'' + \int_{J\setminus I} \vert\Delta\chi\vert \,\diff(\theta_2 - \theta_1) \in [0,\infty].
\end{align}
The right-hand side is well defined in $[0,\infty]$ because $u_{\theta_1} \leq u_{\theta_2}$ on $I$ and $\theta_1(\lbrace b \rbrace) \leq \theta_2(\lbrace b \rbrace)$ for $b \in J\setminus I$.
\end{definition}

If $\theta_1 = \mu$ and $\theta_2 = \nu$, then \eqref{eqn:def:concave integral} coincides with Equation~(4.2) in \cite{BeiglbockNutzTouzi2017} because $\mu$ is concentrated on $I$. As in \cite{BeiglbockNutzTouzi2017}, there is an alternative representation of $(\theta_1-\theta_2)(\chi)$ in terms of an iterated integral with respect to a disintegration of a (one-step) martingale coupling on $\RR^2$:

\begin{lemma}
\label{lem:concave integral:disintegration}
Let $\chi:J\to\RR$ be concave and let $Q\in \cM^d(\theta_1,\theta_2)$. For any disintegration $Q = \theta_1 \otimes \kappa$, we have
\begin{align*}
(\theta_1-\theta_2)(\chi)
&= \int_J \left[ \chi(x_1) - \int_J \chi(x_2) \kappa(x_1,\diff x_2)\right] \theta_1(\diff x_1).
\end{align*}
\end{lemma}

\begin{proof}
The proof of \cite[Lemma~4.1]{BeiglbockNutzTouzi2017} does not use that $\mu \leq_c \nu$ is irreducible. Moreover, for $\bar\chi:J \to\RR$ concave and continuous, the same arguments as in the proof of \cite[Lemma~4.1]{BeiglbockNutzTouzi2017} yield
\begin{align}
\label{eqn:lem:concave integral:disintegration:pf:continuous}
\frac{1}{2}\int_I (u_{\theta_1}-u_{\theta_2}) \dd\bar\chi''
&= \int_J \left[ \bar\chi(x_1) - \int_J \bar\chi(x_2) \,\kappa(x_1,\diff x_2) \right] \theta_1(\diff x_1).
\end{align}
(Note that $\int_{J} \bar\chi(x_2) \,\kappa(x_1,\diff x_2) = \bar\chi(x_1)$ for boundary points $x_1 \in J\setminus I$ because $\kappa$ is a martingale kernel concentrated on $J$.)

For a general concave $\chi:J\to\RR$, we write $\chi = \bar\chi - \vert\Delta\chi\vert\1_{J\setminus I}$ with $\bar\chi$ continuous. Then we can replace $\bar\chi$ with $\chi$ on the left-hand side of \eqref{eqn:lem:concave integral:disintegration:pf:continuous} and the integrand on the right-hand side reads as
\begin{align*}
\chi + \vert\Delta\chi\vert\1_{J\setminus I} - \int_J \chi(x_2)\,\kappa(\cdot,\diff x_2) - \int_{J\setminus I} \vert\Delta\chi(x_2)\vert\,\kappa(\cdot,\diff x_2).
\end{align*}
Integrating this against $\theta_1$ and using Fubini's theorem yields
\begin{align*}
\int_J \left[ \chi(x_1) - \int_J \chi(x_2)\,\kappa(x_1,\diff x_2) \right]\theta(\diff x_1) + \int_{J\setminus I} \vert\Delta\chi\vert \dd\theta_1 - \int_{J\setminus I} \vert\Delta\chi\vert\dd\theta_2.
\end{align*}
Together with \eqref{eqn:lem:concave integral:disintegration:pf:continuous}, this proves the claim.
\end{proof}

It can be shown as in \cite{BeiglbockNutzTouzi2017} that $(\theta_1-\theta_2)(\chi) = \theta_1(\chi) - \theta_2(\chi)$ if at least one of the individual integrals is finite.

We can now define the integral $\theta_1(\varphi) + \theta_2(\psi)$ as in \cite[Definition~4.7]{BeiglbockNutzTouzi2017}.

\begin{definition}
\label{def:BNT integral}
Let $\varphi:J \to\ol\RR$ and $\psi:J\to\ol\RR$ be Borel functions. If there exists a concave function $\chi:J\to\RR$ such that $\varphi-\chi \in L^1(\theta_1)$ and $\psi + \chi \in L^1(\theta_2)$, we say that $\chi$ is a \emph{concave moderator} for $(\varphi,\psi)$ with respect to $\theta_1\leq_c\theta_2$ and set
\begin{align}
\label{eqn:def:BNT integral}
\theta_1(\varphi) + \theta_2(\psi)
&:= \theta_1(\varphi - \chi) + \theta_2(\psi + \chi) + (\theta_1-\theta_2)(\chi) \in (-\infty,\infty].
\end{align}
\end{definition}

As in \cite{BeiglbockNutzTouzi2017}, the expression $\theta_1(\varphi) + \theta_2(\psi)$ defined in \eqref{eqn:def:BNT integral} does not depend on the choice of the concave moderator.

\begin{definition}
\label{def:Lc}
We write $L^c(\theta_1,\theta_2)$ for the space of pairs of Borel functions $\varphi,\psi:J\to\ol\RR$ which admit a concave moderator $\chi$ with respect to $\theta_1\leq_c\theta_2$ such that $(\theta_1-\theta_2)(\chi)<\infty$.
\end{definition}

We next present additional properties of the notions introduced above.

\begin{lemma}
\label{lem:Lc:finite}
Let $(\varphi,\psi) \in L^c(\theta_1,\theta_2)$.
\begin{enumerate}
\item $\varphi$ is finite on atoms of $\theta_1$. If $\varphi$ is concave, then $\varphi < \infty$ on $J$ and $\varphi > -\infty$ on the interior of the convex hull of the support of $\theta_1$.
\item $\psi$ is finite on atoms of $\theta_2$. If $\psi$ is convex, then $\psi > -\infty$ on $J$ and $\psi < \infty$ on the interior of the convex hull of the support of $\theta_2$.
\item If $a,b : \RR \to \RR$ are affine, then $(\varphi + a,\psi+b) \in L^c(\theta_1,\theta_2)$ and
\begin{align*}
\theta_1(\varphi+a) + \theta_2(\psi+b)
&= \lbrace \theta_1(\varphi)+\theta_2(\psi) \rbrace + \theta_1(a) + \theta_2(b).
\end{align*}
\end{enumerate}
\end{lemma}

\begin{proof}
We only prove (iii). Let $\chi$ be a concave moderator for $(\varphi,\psi)$ with respect to $\theta_1\leq_c\theta_2$. Then $\varphi - \chi \in L^1(\theta_1)$, $\psi + \chi \in L^1(\theta_2)$, and $(\theta_1-\theta_2)(\chi) < \infty$. Being affine, $a$ and $b$ are $\theta_1$- and $\theta_2$-integrable. It follows that $\chi$ is also a concave moderator for $(\varphi + a, \psi + b)$ with respect to $\theta_1 \leq_c \theta_2$ and that $(\varphi + a, \psi + b) \in L^c(\theta_1,\theta_2)$. The last assertion is a direct computation.
\end{proof}

\begin{remark}
\label{rem:psi finite}
Recall that $I$ is the interior of the convex hull of the support of $\nu$ and that $J$ is the union of $I$ and any endpoints of $I$ that are atoms of $\nu$. Hence, Lemma~\ref{lem:Lc:finite}~(ii) shows in particular, that if $(\varphi,\psi)\in L^c(\theta_1,\nu)$ with $\psi$ convex, then $\psi$ is finite on $J$.
\end{remark}

We conclude this section with a number of calculation rules for the integrals defined above when $\varphi$ is concave and $\psi$ is convex.

\begin{lemma}
\label{lem:ConvexIntegrationRules}
Let $\mu\leq_c\theta_1\leq_c\theta_2\leq_c\theta_3\leq_c\nu$ (where the pair $\mu\leq_c\nu$ is irreducible) and let $(\varphi,\psi)\in L^c(\theta_1,\theta_3)$ be such that $\varphi$ is concave and finite, $\psi$ is convex and finite, and $\varphi + \psi$ is bounded from below by a concave $\theta_3$-integrable function.
\begin{enumerate}
\item $\varphi$ and $-\psi$ are concave moderators for $(\varphi,\psi)$ with respect to $\theta_1\leq_c\theta_3$.
\item $(\varphi,\psi) \in L^c(\theta_1,\theta_2) \cap L^c(\theta_2,\theta_3)$.
\item $\theta_1(\varphi) + \theta_2(\psi) \leq \theta_1(\varphi) + \theta_3(\psi)$.
\item $\theta_2(\varphi) + \theta_3(\psi) \leq \theta_1(\varphi) + \theta_3(\psi)$.
\end{enumerate}
\end{lemma}

\begin{proof}
Denote by $\xi$ a concave $\theta_3$-integrable lower bound for $\varphi+\psi$. By the concavity of $\xi$, we have $\theta_1(\xi) \geq \theta_2(\xi) \geq \theta_3(\xi)>-\infty$, so that $\xi$ is also $\theta_1$- and $\theta_2$-integrable.

(i): Regarding the concave moderator property of $\varphi$, it suffices to show that $\varphi + \psi$ is $\theta_3$-integrable. Denote by $\varphi'$ the left-derivative of the concave function $\varphi$ on $I$. Then for $(x_1,x_3)\in I\times J$,
\begin{align}
\label{eqn:lem:ConvexIntegrationRules:pf:sandwich1}
\xi(x_3)
&\leq \varphi(x_3) + \psi(x_3)
\leq \varphi(x_1) + \psi(x_3) + \varphi'(x_1)(x_3-x_1).
\end{align}
Fix any $Q \in \cM^d(\theta_1,\theta_3)$. Then \eqref{eqn:lem:ConvexIntegrationRules:pf:sandwich1} also holds $Q$-a.e.~on $J\times J$ (setting $\varphi' = 0$ on $J\setminus I$ for example); this uses that any mass in a point of $J \setminus I$ stays put under a martingale transport plan. Since $\xi$ is $\theta_3$-integrable, the negative part of the right-hand side in \eqref{eqn:lem:ConvexIntegrationRules:pf:sandwich1} is $Q$-integrable. Then it can be argued as in \cite[Remark~4.10]{BeiglbockNutzTouzi2017} that the right-hand side in \eqref{eqn:lem:ConvexIntegrationRules:pf:sandwich1} is $Q$-integrable. It follows that $\varphi+\psi$ is $\theta_3$-integrable.

Regarding the assertion about $-\psi$, it suffices to show that $\varphi + \psi$ is $\theta_1$-integrable. We have 
\begin{align}
\label{eqn:lem:ConvexIntegrationRules:pf:sandwich2}
\begin{split}
\xi(x_1)
&\leq \varphi(x_1) + \psi(x_1)\\
&= [\varphi(x_1) + \psi(x_3) + \varphi'(x_1)(x_3-x_1)] \\
&\qquad+ [\psi(x_1) - \psi(x_3) - \varphi'(x_1)(x_3-x_1)]\quad Q\text{-a.e.~on }J\times J.
\end{split}
\end{align}
By the above, the first term on the right-hand side is $Q$-integrable. Thus, the negative part of the second term is also $Q$-integrable. Hence, we may integrate the second term iteratively using Fubini's theorem as in \cite[Remark~4.10]{BeiglbockNutzTouzi2017}. The $Q$-integral equals $-(\theta_1-\theta_3)(-\psi) \leq 0$. In particular, the right-hand side in \eqref{eqn:lem:ConvexIntegrationRules:pf:sandwich2} is $Q$-integrable. It follows that $\varphi+\psi$ is $\theta_1$-integrable.

(ii)--(iv): We know from the above that $\varphi+\psi$ is $\theta_3$-integrable. It follows that $\varphi$ is a concave moderator for $(\varphi,\psi)$ with respect to $\theta_2\leq_c\theta_3$. Because $\theta_1\leq_c\theta_2$, we have that $u_{\theta_1} \leq u_{\theta_2}$ and $\theta_1(\lbrace b \rbrace) \leq \theta_2(\lbrace b \rbrace)$ for $b \in J \setminus I$. Thus, $(\theta_2 - \theta_3)(\varphi) \leq (\theta_1-\theta_3)(\varphi) < \infty$ (cf.~Definition~\ref{def:concave integral}). Hence, $(\varphi,\psi) \in L^c(\theta_2,\theta_3)$ and 
\begin{align*}
\theta_2(\varphi) + \theta_3(\psi)
&= \theta_2(\varphi-\varphi) + \theta_3(\varphi+\psi) + (\theta_2-\theta_3)(\varphi)\\
&\leq \theta_1(\varphi-\varphi) + \theta_3(\varphi+\psi) + (\theta_1-\theta_3)(\varphi)\\
&= \theta_1(\varphi) + \theta_3(\psi).
\end{align*}
One can show similarly that $(\varphi,\psi)\in L^c(\theta_1,\theta_2)$ and that $\theta_1(\varphi) + \theta_2(\psi) \leq \theta_1(\varphi) + \theta_3(\psi)$.
\end{proof}

\subsection{Pathwise stochastic integration}
\label{sec:stochastic integration}

For any $\FF$-adapted c\`adl\`ag process $H$ of finite variation, the integral $H_-\bullet X_T = \int_{(0,T]} H_{t-} \dd X_t$ can be defined \emph{pathwise}, i.e., for each $\omega \in \Omega$ individually, via integration by parts as follows:
\begin{align}
\label{eqn:integral:ibp}
H_-\bullet X_T
&:= X_T H_T - X_0 H_0 - \int_{(0,T]} X_t \dd H_t,
\end{align}
where the integral on the right-hand side is the pathwise Lebesgue--Stieltjes integral. Setting $H_{0-} = 0$, so that $\Delta H_0 = H_0$, we can recast  \eqref{eqn:integral:ibp} as
\begin{align}
\label{eqn:integral:ibp:alternative}
H_-\bullet X_T
&= (X_T-X_0)H_0 + \int_{(0,T]} (X_T - X_t) \dd H_t.
\end{align}
For any martingale measure $P$, if the (standard) stochastic integral of $H_-$ with respect to $X$ exists, then it is $P$-indistinguishable from the pathwise stochastic integral.

We need to give a sensible meaning to the integral $H_-\bullet X_T$ for certain integrands $H$ which are not necessarily of finite variation, but may diverge in finite time.

\begin{example}
\label{ex:integral}
The following example motivates our extension of the pathwise stochastic integral for finite variation integrands. Let $\mu = \delta_0$ and $\nu = \frac{1}{2}\delta_{-1} + \frac{1}{2}\delta_1$. Then $\mu \leq_c \nu$ are irreducible with domain $(I,J) = ((-1,1),[-1,1])$. Consider a payoff function $f$ which is convex on $[-1,1]$ and has infinite (one-sided) derivatives at $-1$ and $1$, e.g., $f(x) = 1-\sqrt{1-x^2}\1_{[-1,1]}(x)$,  A semi-static superhedge for the Asian-style derivative $f(\frac{1}{T}\int_0^T X_t \dd t)$ can be derived as follows. By Jensen's inequality and the convexity of $f$, for every path of $X$ that evolves in $[-1,1]$,
\begin{align*}
f\Big(\frac{1}{T}\int_0^T X_t \dd t\Big)
&\leq \int_0^T f(X_t) \,\frac{\dd t}{T}
\leq \int_0^T \big(f(X_T) - f'(X_t)(X_T-X_t)\big) \,\frac{\diff t}{T}\\
&= f(X_T) - \int_0^T (X_T-X_t) f'(X_t)\,\frac{\diff t}{T}.
\end{align*}
Comparing this with \eqref{eqn:integral:ibp:alternative}, a semi-static superhedge for the Asian-style derivative is obtained from a European-style derivative with payoff $f(X_T)$ maturing at $T$ and a dynamic trading strategy $H$ with $H_0 = 0$ and dynamics $\diff H_t = -f'(X_t) \frac{\diff t}{T}$. Then $H$ is of finite variation whenever $X$ stays away from the boundaries of $(-1,1)$. But, as $X$ approaches $-1$ or $1$, the derivative $f'(X_t)$ becomes arbitrarily large (in absolute value), and $H$ may fail to be of finite variation. It turns out, however, that the integral $\int_0^T (X_T-X_t)f'(X_t) \, \frac{\diff t}{T}$ is still well defined on these paths. The reason is that when paths of $X$ come arbitrarily close to $1$, say, then for any martingale coupling $P\in\cM(\mu,\nu)$, $X_T=1$ $P$-a.s.~on these paths (because $J = [-1,1]$), so that $X_T-X_t$ becomes small and counteracts the growth of $f'(X_t)$.
\end{example}

We shall define a pathwise stochastic integral for $\FF$-adapted c\`adl\`ag integrators $X$ and integrands $\hat H_-$ of the form
\begin{align}
\label{eqn:integral:integrand}
\hat H_t
&= h_0 + \int_{(0,t]} h_s \dd Y_s
\end{align}
for an $\FF$-adapted c\`adl\`ag process $Y = (Y_t)_{t\in[0,T]}$ of finite variation and an $\FF$-adapted process $h = (h_t)_{t\in[0,T]}$---even in certain situations where the right-hand side of \eqref{eqn:integral:integrand} is not finite. The idea is to formally substitute \eqref{eqn:integral:integrand} into \eqref{eqn:integral:ibp:alternative}, formally use the associativity of Lebesgue--Stieltjes integrals, and then employ the resulting expression as a definition for a pathwise stochastic integral. We first introduce a set of integrands for this integral.

\begin{definition}
\label{def:integrand}
Let $\Omega' \subset D([0,T];\RR)$. We denote by $L(\Omega')$ the set of pairs $(h,Y)$ consisting of an $\FF$-adapted process $h$ and an $\FF$-adapted c\`adl\`ag process $Y$ of finite variation such that the process $((X_T - X_t)h_t)_{t\in[0,T]}$ is $\diff Y$-integrable on $(0,T]$ for each path in $\Omega'$.
\end{definition}

If $Y$ is an $\FF$-adapted c\`adl\`ag process of finite variation, then $(1,Y) \in L(\Omega')$ for any $\Omega' \subset D([0,T];\RR)$ (because any c\`adl\`ag function is bounded on compact intervals).

We fix a set $\Omega' \subset D([0,T];\RR)$ for the rest of this section.

\begin{definition}
\label{def:integral}
For $H=(h,Y) \in L(\Omega')$, we set
\begin{align}
\label{eqn:def:integral}
H \diamond X_T
&:= (X_T-X_0)h_0 + \int_{(0,T]} (X_T - X_t) h_t \dd Y_t\quad\text{on}\quad \Omega'.
\end{align}
\end{definition}

We note that the Lebesgue--Stieltjes integral on the right-hand side of \eqref{eqn:def:integral} is well defined and finite by the definition of $L(\Omega')$. The following result shows that for pathwise bounded $h$, $H\diamond X_T$ coincides with $\hat H_-\bullet X_T$ for $\hat H$ as in \eqref{eqn:integral:integrand}. This motivates the interpretation of $H\diamond X_T$ as the gains from trading in $X$ according to a self-financing  trading strategy $\hat H_-$.

\begin{proposition}
\label{prop:integral:consistency}
Let $H = (h,Y) \in L(\Omega')$ and $\omega \in \Omega'$. If the function $t\mapsto h_t(\omega)$ is bounded on $[0,T]$, then 
\begin{align*}
(H \diamond X_T)(\omega)
&= (\hat H_- \bullet X_T)(\omega),
\end{align*}
where $\hat H = h_0 + \int_{(0,\cdot]} h \dd Y$.
\end{proposition}

If we set $h_0 = Y_0$ and $h_t = 1$ for $t \in (0,T]$ for an $\FF$-adapted c\`adl\`ag process $Y$ of finite variation, then $H=(h,Y)\in L(\Omega)$ and by Proposition~\ref{prop:integral:consistency},
\begin{align*}
H \diamond X_T
&= Y_- \bullet X_T\quad\text{on}\quad\Omega.
\end{align*}
So the integral $H\diamond X_T$ embeds all pathwise stochastic integrals $Y_-\bullet X_T$.

\begin{proof}[Proof of Proposition~\ref{prop:integral:consistency}]
Since $h(\omega)$ is bounded on $[0,T]$, $\hat H_t(\omega) = h_0(\omega) + \int_{(0,t]} h_s(\omega) \dd Y_s(\omega)$ is a well-defined c\`adl\`ag finite variation function on $[0,T]$. Thus, by \eqref{eqn:integral:ibp:alternative},
\begin{align*}
\hat H_- \bullet X_T
&= (X_T - X_0)h_0 + \int_{(0,T]} (X_T-X_s) h_s \dd Y_s
= H \diamond X_T.\qedhere
\end{align*}
\end{proof}

\section{Robust pricing and superhedging problems}
\label{sec:problems}

Throughout this section, we fix an irreducible pair $\mu \leq_c \nu$ with domain $(I,J)$ and a Borel function $f:\RR\to\ol\RR$ which is bounded from below by a $\nu$-integrable concave function.

\subsection{Pricing problem}
\label{sec:pricing problem}

Our pricing and hedging duality applies to a wide range of exotic derivatives including American options, fixed strike Asian options, Bermudan options, and European options with an intermediate maturity. We now describe this class of derivative securities. 

\begin{definition}
A nonnegative $\FF$-adapted nondecreasing c\`adl\`ag process $A=(A_t)_{t\in[0,T]}$ is called an \emph{averaging process} if $A_T(\omega)=1$ for every $\omega\in\Omega$. If in addition $A_0(\omega) = 0$ and $\Delta A_T(\omega) = 0$ for each $\omega \in \Omega$, then $A$ is called an \emph{interior averaging process}. If in addition there is $t\in(0,T)$ such that $A_t(\omega) = 0$ for each $\omega\in\Omega$, then $A$ is called a \emph{strictly interior averaging process}.
\end{definition}

Recall that we set $A_{0-} = 0$ and note that for each $\omega\in\Omega$, $A(\omega)$ can be identified with a Borel probability measure on $[0,T]$. If $A$ is an interior averaging process, then this probability measure is supported on $(0,T)$, and if $A$ is a strictly interior averaging process then its support is (uniformly in $\omega)$ contained in $[t,T)$ for some $t \in(0,T)$.

Given a nonempty set $\cA$ of averaging processes, we consider a derivative security whose payoff at time $T$ is
\begin{align}
\label{eqn:payoff}
f\Big(\int_{[0,T]}X_t \dd A_t\Big),
\end{align}
where $A \in \cA$ is chosen by the buyer and the seller observes $(A_s)_{s\in[0,t]}$ at time $t$. Then the \emph{robust model-based price} is defined as
\begin{align}
\label{eqn:pricing problem}
\bfS_{\mu,\nu}(f,\cA)
&=\sup_{P \in \cM(\mu,\nu)} \sup_{A \in \cA} \EX[P]{f\Big(\int_{[0,T]} X_t \dd A_t\Big)}.
\end{align}
In other words, $\bfS_{\mu,\nu}(f,\cA)$ is the highest model-based price of the derivative security \eqref{eqn:payoff} among all martingale models which are consistent with the given marginal distributions.

\begin{remark}
One can show that for each $P \in \cM(\mu,\nu)$ and each averaging process $A$, the law of $\int_{[0,T]} X_t \dd A_t$ under $P$ is in convex order between $\mu$ and $\nu$; cf.~Lemma~\ref{lem:average:convex order}. Because $f$ is by assumption bounded from below by a $\nu$-integrable concave function, the expectations in \eqref{eqn:pricing problem} are well defined.
\end{remark}

Important special cases are obtained for specific choices of $\cA$.

\begin{example}[No special exercise rights]
\label{ex:asian european}
Setting $\cA = \lbrace A \rbrace$ deprives the buyer of any special exercise rights and reduces \eqref{eqn:pricing problem} to the more familiar robust pricing problem
\begin{align*}
\sup_{P\in\cM(\mu,\nu)} \EX[P]{F}
\end{align*}
for the derivative security $F = f(\int_{[0,T]}X_t\dd A_t)$.

\begin{enumerate}
\item \emph{Asian options.} Setting $A_t = t/T$ recovers the Asian-style derivative $f(\frac{1}{T}\int_0^T X_t \dd t)$; this includes fixed strike Asian puts and calls, but not floating strike Asian options. This robust pricing problem is analyzed in \cite{CoxKallblad2017}.

\item \emph{European options.} Setting $A_t = \1_{[T',T]}(t)$ yields a European-style payoff $f(X_{T'})$ with an intermediate maturity $T' \in (0,T)$.
\end{enumerate}
\end{example}

\begin{example}[Special exercise rights]~
\label{ex:american bermudan}
Fix a nonempty set $\cT$ of $[0,T]$-valued $\FF$-stopping times, and consider $\cA = \lbrace \1_{\ldbrack \tau, T \rdbrack}:\tau \in \cT\rbrace$. Then \eqref{eqn:pricing problem} reduces to
\begin{align}
\label{eqn:ex:american bermudan}
\sup_{P\in\cM(\mu,\nu)} \sup_{\tau\in\cT}\EX[P]{f(X_\tau)}.
\end{align}
\begin{enumerate}
\item \emph{American options.} If $\cT$ consists of all $[0,T]$-valued $\FF$-stopping times, then \eqref{eqn:ex:american bermudan} is the robust American option pricing problem analyzed in \cite{BayraktarCoxStoev2018}.

\item \emph{Bermudan options.} Bermudan options with exercise dates $0\leq T_1 < \cdots < T_n \leq T$ are covered by choosing $\cT$ to be the set of $\lbrace T_1,\ldots, T_n \rbrace$-valued $\FF$-stopping times. 
\end{enumerate}
\end{example}

\subsection{Superhedging problem}
\label{sec:superhedging problem}

In the case of robust semi-static superhedging of American options, it is well known that a pricing-hedging duality can in general only hold if the seller of the option can adjust the dynamic part of his trading strategy after the option has been exercised; cf.~\cite[Section~3]{BayraktarHuangZhou2015}. In other words, the buyer has to communicate her decision of exercising to the seller at the time of exercising. The analog in our setting is that the seller observes $A_t$ at time $t$. That is, his trading strategy can be ``adapted'' to the averaging process chosen by the buyer.

To make this precise, let $\hat\Omega$ be the cartesian product of $\Omega$ and the set of nonnegative, nondecreasing, c\`adl\`ag functions $a:[0,T]\to[0,1]$ with $a(T) = 1$. As $\hat\Omega$ is a subspace of the Skorokhod space $D([0,T];\RR\times[0,1])$, we can equip it with the subspace Skorokhod topology and denote by $\hat\cF$ the corresponding Borel $\sigma$-algebra. We write $\hat\FF = (\hat\cF_t)_{t\in[0,T]}$ for the (raw) filtration generated by the canonical process on $\hat\Omega$. For any process $Z$ on $\hat\Omega$ and any averaging process $A$ (on $\Omega$), we define the process $Z^A$ on $\Omega$ by
\begin{align*}
Z^A_t(\omega)
&= Z_t(\omega,A(\omega)), \quad \omega\in\Omega.
\end{align*}
Note that if $Z$ is $\hat\FF$-adapted, then $Z^A$ is $\FF$-adapted, and if $Z$ is c\`adl\`ag or of finite variation, then so is $Z^A$.

Next, we define a suitable set of paths for the hedging problem. Let $\Omega_{\mu,\nu}\subset\Omega$ denote the subset of paths which start in $I$, evolve in $J$, and are ``captured'' if they approach the boundary $\partial J$:
\begin{align}
\label{eqn:relevant paths}
\Omega_{\mu,\nu}
:= \lbrace \omega \in \Omega : 
&\;\,\text{$\omega_0 \in I$, $\omega_t \in J$ for all $t \in (0,T]$,}\notag\\
&\;\,\text{if $\omega_{t-}\in \partial J$, then $\omega_u =\omega_{t-}$ for all $u\in[t,T]$, and}\\
&\;\,\text{if $\omega_t \in  \partial J$, then $\omega_u = \omega_t$ for all $u\in[t,T]$} \rbrace.\notag
\end{align}

One can show that every martingale coupling between $\mu$ and $\nu$ is concentrated on $\Omega_{\mu,\nu}$:\footnote{The fact that $\mu$ and $\nu$ are concentrated on $I$ and $J$, respectively, together with the martingale property implies that $P$-a.e.~path has the first two properties in \eqref{eqn:relevant paths}. The other two properties can be shown similarly to the fact that nonnegative supermartingales are almost surely captured in zero (cf.~\cite[Lemma~7.31]{Kallenberg2002}).}

\begin{lemma}
\label{lem:relevant paths}
$\Omega_{\mu,\nu} \in \cF$ and $P[\Omega_{\mu,\nu}] = P[\Omega]$ for every $P\in\cM(\mu,\nu)$.
\end{lemma}

We are now ready to define the trading strategies for the robust superhedging problem. 

\begin{definition}
\label{def:strategy}
A \emph{semi-static trading strategy} is a triplet $(\varphi,\psi,H)$ consisting of a pair of functions $(\varphi,\psi) \in L^c(\mu,\nu)$ and a pair $H=(h_t,Y_t)_{t\in[0,T]}$ of $\hat\FF$-adapted processes on $\hat\Omega$ such that
\begin{align}
\label{eqn:def:strategy:pathwise integrability}
H^A
&:= (h^A,Y^A) \in L(\Omega_{\mu,\nu})
\quad\text{for every averaging process }A.
\end{align}
\end{definition}

The portfolio value at time $T$ of a semi-static trading strategy is given by the sum of the static part with payoffs $\varphi(X_0)$ and $\psi(X_T)$ and the gains $H^A\diamond X_T$ of the dynamic part:
\begin{align}
\label{eqn:portfolio value:terminal}
\varphi(X_0) + \psi(X_T) + H^A\diamond X_T.
\end{align}
The initial cost to set up this position is equal to the initial price of the static part:
\begin{align}
\label{eqn:portfolio value:initial}
\mu(\varphi) + \nu(\psi).
\end{align}

We now turn our attention to semi-static trading strategies which dominate the payoff \eqref{eqn:payoff} of our derivative security for  each path in $\Omega_{\mu,\nu}$ and every averaging process in $\cA$.

\begin{definition}
\label{def:superhedge}
A semi-static trading strategy $(\varphi,\psi,H)$ is called a \emph{semi-static superhedge (for $f$ and $\cA$)} if for every $A \in \cA$,
\begin{align}
\label{eqn:def:superhedge}
f\Big(\int_{[0,T]} X_t \dd A_t\Big)
&\leq \varphi(X_0) + \psi(X_T) + H^A \diamond X_T\quad\text{on}\quad\Omega_{\mu,\nu}
\end{align}
and
\begin{align}
\label{eqn:def:superhedge:admissibility condition}
\EX[P]{\varphi(X_0) + \psi(X_T) + H^A \diamond X_T}
&\leq \mu(\varphi) + \nu(\psi), \quad P \in \cM(\mu,\nu).
\end{align}
The set of semi-static superhedges for $f$ and $\cA$ is denoted by $\cD_{\mu,\nu}(f,\cA)$.
\end{definition}

The requirement \eqref{eqn:def:superhedge:admissibility condition} is an admissibility condition. It demands that for every $P\in\cM(\mu,\nu)$, the portfolio value, consisting of both the static and the dynamic part, is a one-step $P$-supermartingale between the time at which the static part is set up and time $T$. In other words, the expectation of the terminal portfolio value \eqref{eqn:portfolio value:terminal} is less than or equal to the initial portfolio value \eqref{eqn:portfolio value:initial}.

We define the \emph{robust superhedging price (for $f$ and $\cA$)} as the ``minimal'' initial capital required to set up a semi-static superhedge for $f$ and $\cA$:\footnote{We use the convention $\inf\emptyset=\infty$.}
\begin{equation}
\label{eqn:hedging problem}
\bfI_{\mu,\nu}(f,\cA)
= \inf_{(\varphi,\psi,H)\in\cD_{\mu,\nu}(f,\cA)} \lbrace \mu(\varphi) + \nu(\psi) \rbrace.
\end{equation}

\subsection{Weak and strong duality}
\label{sec:duality}

Weak duality between the robust pricing and hedging problems is an immediate consequence of their definitions:

\begin{lemma}[Weak duality]
\label{lem:weak duality}
Let $f:\RR\to\ol\RR$ be Borel and bounded from below by a $\nu$-integrable concave function and let $\cA$ be a nonempty set of averaging processes. Then
\begin{equation*}
\bfS_{\mu,\nu}(f,\cA)
\leq \bfI_{\mu,\nu}(f,\cA).
\end{equation*}
\end{lemma}

\begin{proof}
Let $P \in \cM(\mu,\nu)$, $A\in\cA$, and $(\varphi,\psi,H) \in \cD_{\mu,\nu}(f,\cA)$ (there is nothing to show if this set is empty). Taking $P$-expectations in \eqref{eqn:def:superhedge} and using \eqref{eqn:def:superhedge:admissibility condition} shows that $\EX[P]{f(\int_{[0,T]} X_t \dd A_t)} \leq \mu(\varphi) + \nu(\psi)$.

This proves the claim as $P$, $A$, and $(\varphi,\psi,H)$ were arbitrary.
\end{proof}

With an additional mild assumption on either $\cA$ or $f$, we obtain strong duality and the existence of dual minimizers:

\begin{theorem}
\label{thm:StrongDuality}
Let $\mu \leq_c \nu$ be irreducible, let $f:\RR\to[0,\infty]$ be Borel, and let $\cA$ be a set of averaging processes. Suppose that one of the following two conditions holds:
\begin{itemize}
\item $f$ is lower semicontinuous and $\cA$ contains an interior averaging process;
\item $\cA$ contains a strictly interior averaging process.
\end{itemize}
Then
\begin{align*}
\bfS_{\mu,\nu}(f,\cA)
&= \bfI_{\mu,\nu}(f,\cA) \in [0,\infty]
\end{align*}
and this value is independent of $\cA$ as long as one of the two conditions above holds. Moreover, if $\bfI_{\mu,\nu}(f,\cA) < \infty$, then there exists an optimizer $(\varphi,\psi,H)\in\cD_{\mu,\nu}(f,\cA)$ for $\bfI_{\mu,\nu}(f,\cA)$.
\end{theorem}

\begin{remark}\quad
\label{rem:strong duality:derivatives}
\begin{enumerate}
\item For fixed $f$, the robust model-based price $\bfS_{\mu,\nu}(f,\cA)$ is invariant under the choice of the set $\cA$ (as long as the assumptions of Theorem~\ref{thm:StrongDuality} hold). In particular, American, Bermudan, and European options with intermediate maturity (cf.~Examples~\ref{ex:asian european}--\ref{ex:american bermudan}) all have the same robust model-based price (because the corresponding sets $\cA$ all contain a strictly interior averaging process). If $f$ is lower semicontinuous, this extends to the Asian-style option of Example~\ref{ex:asian european}~(i). If more than two marginals are given, then the robust model-based prices of these derivatives typically differ; see Example~\ref{ex:multi marginal}.

\item Derivatives of the form \eqref{eqn:payoff} that depend distinctly on $X_0$ and/or $X_T$ such as $f(\frac{1}{2}(X_0 + X_T))$ are not covered by Theorem~\ref{thm:StrongDuality} ($\cA$ does not contain an \emph{interior} averaging process). In these cases, the robust model-based price is still bounded above by the corresponding robust model-based price of, say, the European-style derivative $f(X_{T/2})$. However, the inequality is typically strict; see Example~\ref{ex:dual auxiliary fail}.
\end{enumerate}
\end{remark}

\begin{remark}\quad
\label{rem:strong duality:technical}
\begin{enumerate}
\item Theorem~\ref{thm:StrongDuality} can be extended to non-irreducible marginals along the lines of \cite[Section~7]{BeiglbockNutzTouzi2017}.

\item Strong duality continues to hold if we restrict ourselves to finite variation strategies; cf.~Remark~\ref{rem:strong duality with FV strategies} for an outline of the argument. It is an open question whether there is (in general) a dual minimizer $(\varphi,\psi,H)$ with a dynamic part $H$ of finite variation.
\end{enumerate}
\end{remark}

We defer the proof of Theorem~\ref{thm:StrongDuality} to the end of Section~\ref{sec:auxiliary duality}. The idea is as follows. We bound the pricing problem from below and the hedging problem from above by auxiliary maximization and minimization problems, respectively, and show that strong duality holds between those two auxiliary problems. Then all four problems have equal value and in particular strong duality for the pricing and hedging problems holds. Moreover, we show that the auxiliary dual problem admits a minimizer and that every element in the dual space of the auxiliary problem gives rise to a semi-static superhedge with the same cost. Then, in particular, the minimizer of the auxiliary dual problem yields an optimal semi-static superhedge for $f$ and $\cA$ (which is independent of $\cA$).

\section{Auxiliary problems}
\label{sec:auxiliary problems}

Throughout this section, we fix an irreducible pair $\mu \leq_c \nu$ with domain $(I,J)$ and a function $f:\RR\to\ol\RR$ which is bounded from below by a $\nu$-integrable concave function.

The auxiliary primal and dual problems are formally derived in Section~\ref{sec:motivation}. They are rigorously introduced in Sections~\ref{sec:primal problem}--\ref{sec:dual problem} and proved to be lower and upper bounds of the robust model-based price and the robust superhedging price, respectively. Their strong duality is proved in Section~\ref{sec:auxiliary duality}. Finally, structural properties of primal and dual optimizers of the auxiliary problems are studied in Section~\ref{sec:structure}.

\subsection{Motivation}
\label{sec:motivation}

The key property of payoffs of the form \eqref{eqn:payoff} is that the law of $\int_{[0,T]}X_t\dd A_t$ under $P\in\cM(\mu,\nu)$ is in convex order between $\mu$ and $\nu$. In this section, we explain this observation and how it can be used to estimate the robust pricing problem from below and the robust superhedging problem from above.

Let $P \in \cM(\mu,\nu)$ and let $\tau$ be a $[0,T]$-valued $\FF$-stopping time. An application of the optional stopping theorem and Jensen's inequality shows that for any convex function $\psi$,
\begin{equation*}
\mu(\psi)
= \EX[P]{\psi(X_0)}
= \EX[P]{\psi(\cEX[P]{X_\tau}{\cF_0})}
\leq \EX[P]{\psi(X_\tau)}\qquad \text{and}
\end{equation*}
\begin{equation*}
\nu(\psi)
= \EX[P]{\psi(X_T)}
\geq \EX[P]{\psi(\cEX[P]{X_T}{\cF_\tau})}
= \EX[P]{\psi(X_\tau)},
\end{equation*}
so that the law of $X_\tau$ under $P$ is in convex order between $\mu$ and $\nu$.

Using a time change argument and again Jensen's inequality and the optional stopping theorem, it can be shown that this property generalizes to the random variable $\int_{[0,T]}X_t \dd A_t$ for an averaging process $A$.

\begin{lemma}
\label{lem:average:convex order}
Let $P\in\cM(\mu,\nu)$ and let $A$ be an averaging process. Then the law of $\int_{[0,T]}X_t\dd A_t$ under $P$ is in convex order between $\mu$ and $\nu$.
\end{lemma}

In the sequel, we write $\bfS = \bfS_{\mu,\nu}(f,\cA)$ and $\bfI = \bfI_{\mu,\nu}(f,\cA)$ for brevity. Lemma~\ref{lem:average:convex order} implies that
\begin{equation*}
\bfS
\leq \sup_{\mu\leq_c\theta\leq_c\nu}\theta(f)
=: \wt\bfS.
\end{equation*}
We show in Section~\ref{sec:primal problem} that also the converse inequality holds under mild assumptions on $f$ and $\cA$. Thus, $\bfS=\wt\bfS$ and one is led to expect that $\bfI = \wt\bfI$ for a suitable dual problem $\wt\bfI$ to $\wt\bfS$.

Let us thus formally derive the Lagrange dual problem for $\wt\bfS$. Dualizing the constraint $\mu\leq_c\theta\leq_c\nu$ suggests to consider the Lagrangian
\begin{align}
\label{eqn:motivation:lagrangian}
L(\theta,\psi_1,\psi_2)
&:= \theta(f) + (\theta(\psi_1)-\mu(\psi_1)) + (\nu(\psi_2)-\theta(\psi_2)), 
\end{align}
where \emph{convex} functions $\psi_1,\psi_2$ are taken as Lagrange multipliers.\footnote{Note that the last two terms in \eqref{eqn:motivation:lagrangian} are nonnegative for all convex $\psi_1,\psi_2$ if and only if the primal constraint $\mu\leq_c\theta\leq_c\nu$ holds.} Then the Lagrange dual problem is
\begin{equation*}
\wt\bfI
=\inf_{\psi_1,\psi_2} \sup_{\theta} L(\theta,\psi_1,\psi_2)
= \inf_{\psi_1,\psi_2} \sup_{\theta} \left\lbrace \theta(f+\psi_1-\psi_2) -\mu(\psi_1) + \nu(\psi_2)\right\rbrace
\end{equation*}
where the infima are taken over convex functions and the suprema are taken over finite measures. Viewing the finite measure $\theta$ as a Lagrange multiplier for the constraint $f \leq -\psi_1 + \psi_2$ and relabeling $\varphi = -\psi_1$ and $\psi = \psi_2$, we obtain
\begin{equation}
\label{eqn:motivation:Lagrange dual}
\wt\bfI
= \inf\lbrace \mu(\varphi)+\nu(\psi) : \text{$\varphi$ concave, $\psi$ convex, and $\varphi+\psi \geq f$} \rbrace.
\end{equation}
In the precise definition of $\wt\bfI$ in Section~\ref{sec:dual problem}, $\mu(\varphi)+\nu(\psi)$ is understood in the generalized sense of Definition~\ref{def:BNT integral} and the inequality $\varphi + \psi \geq f$ is required to hold on $J$. We then show that each feasible element $(\varphi,\psi)$ for $\wt\bfI$ entails an element $(\varphi,\psi,H) \in \cD_{\mu,\nu}(f,\cA)$ (Proposition~\ref{prop:dynamic part}), which implies that $\bfI\leq\wt\bfI$.

Combining the above with the weak duality inequality (Lemma~\ref{lem:weak duality}) yields
\begin{align*}
\wt\bfS
&= \bfS
\leq \bfI
\leq \wt\bfI.
\end{align*}
Hence, strong duality and dual attainment for the robust pricing and superhedging problems reduce to the same assertions for the simpler auxiliary problems, which are proved in Section~\ref{sec:auxiliary duality}.

\subsection{Auxiliary primal problem}
\label{sec:primal problem}

Consider the \emph{auxiliary primal problem}
\begin{equation}
\label{eqn:auxiliary primal problem}
\wt\bfS_{\mu,\nu}(f)
= \sup_{\mu\leq_c \theta \leq_c \nu} \theta(f),
\end{equation}
where $\theta(f)$ is understood as the outer integral if $f$ is not Borel-measurable. Under suitable conditions on $f$ and $\cA$, the primal value $\wt\bfS_{\mu,\nu}(f)$ is a lower bound for the robust model-based price \eqref{eqn:pricing problem}:

\begin{proposition}
\label{prop:reduction:primal}
Let $\cA$ be a set of averaging processes. Suppose that one of the following two sets of conditions holds:
\begin{enumerate}
\item $\cA$ contains an interior averaging process and $f$ is lower semicontinuous and bounded from below by a $\nu$-integrable concave function $\varphi:J\to\ol\RR$;
\item $\cA$ contains a strictly interior averaging process and $f$ is Borel.
\end{enumerate}
Then
\begin{equation*}
\wt\bfS_{\mu,\nu}(f)
\leq \bfS_{\mu,\nu}(f,\cA).
\end{equation*}
\end{proposition}

The proof of Proposition~\ref{prop:reduction:primal} is given at the end of this section. It is based on the following construction of measures in $\cM(\mu,\nu)$ under which the law of $\int_{[0,T]} X_t \dd A_t$ equals (approximately or exactly) a given $\theta$. This construction also highlights the importance of $\cA$ containing an \emph{interior} averaging process, which does not put any mass on the times $0$ and $T$ at which the marginal distributions of $X$ are given; see Example~\ref{ex:dual auxiliary fail} for a counterexample.

\begin{lemma}
\label{lem:average:weak approximation}
Let $\mu \leq_c \theta \leq_c \nu$.
\begin{enumerate}
\item There is a sequence $(P_n)_{n\geq1}\subset\cM(\mu,\nu)$ such that
\begin{align*}
\cL^{P_n}\Big(\int_{[0,T]}X_t \dd A_t\Big) \xrightarrow{n\to\infty} \theta \quad\text{weakly}
\end{align*}
for every interior averaging process $A$.
\item If $A$ is a strictly interior averaging process, then there is $P \in \cM(\mu,\nu)$ (depending on $A$) such that $\cL^P\big(\int_{[0,T]}X_t \dd A_t\big) = \theta$.
\end{enumerate}
\end{lemma}

\begin{proof}
(i): By the two-step adaptation of Proposition~\ref{prop:Strassen}, there exists a measure $Q \in \cM^d(\mu,\theta,\nu)$. For all $n$ large enough, let $\iota^n:\RR^3 \to \Omega$ be the embedding of $\RR^3$ in $\Omega$ which maps $(y_1,y_2,y_3)$ to the piecewise constant path
\begin{align}
\label{eqn:lem:average:weak approximation:pf:embedding}
[0,T]\ni t \mapsto y_1\1_{[0,\frac{1}{n})}(t) + y_2\1_{[\frac{1}{n},T)}(t) + y_3\1_{\{T\}}(t)
\end{align}
(which jumps (at most) at times $\frac{1}{n}$ and $T$), and denote by $P_n := Q \circ (\iota^n)^{-1}$ the associated pushforward measure. Then $P_n \in \cM(\mu,\nu)$ by the corresponding properties of $Q$. Moreover, denoting the canonical process on $\RR^3$ by $(Y_1,Y_2,Y_3)$ and setting $A^n = A \circ \iota^n$ for an interior averaging process $A$, we have
\begin{align} \label{eqn:lem:average:weak approximation:pf:difference}
\int_{[0,T]} (\iota^n)_t \dd A^n_t - Y_2
&= Y_1 A^n_{\frac{1}{n}-} + Y_2 (A^n_{T-}-A^n_{\frac{1}{n}-}) + Y_3 \Delta A^n_T -Y_2A_T \nonumber\\
&= (Y_1 - Y_2)A^n_{\frac{1}{n}-} +(Y_3 - Y_2)\Delta A^n_T \\
&= (Y_1 - Y_2)A^n_{\frac{1}{n}-}\quad\text{on}\quad\RR^3, \nonumber
\end{align}
where we use the properties $A_T = 1$ and $\Delta A_T=0$ of an interior averaging process.

By construction, the law of $\int_{[0,T]} (\iota^n)_t \dd A^n_t$ under $Q$ coincides with the law of $\int_{[0,T]} X_t \dd A_t$ under $P_n$ and the law of $Y_2$ under $Q$ is $\theta$. It thus suffices to prove that the right-hand side in \eqref{eqn:lem:average:weak approximation:pf:difference} converges to zero in $L^1(Q)$ as $n\to\infty$. To this end, note that $|Y_1 - Y_2| \leq |Y_1| + |Y_2|$ is $Q$-integrable because $\mu$ and $\theta$ have finite first moments. Thus, by dominated convergence, it is enough to show that $A^n_{\frac{1}{n}-} \to 0$ pointwise as $n\to0$. So fix $(y_1,y_2,y_3) \in \RR^3$. Since $A$ is $\FF$-adapted, $A_{\frac{1}{n}-}(\omega)$ only depends on the values of the path $\omega$ on the interval $[0,\frac{1}{n})$. In view of the embedding \eqref{eqn:lem:average:weak approximation:pf:embedding}, this means that 
\begin{align*}
A^n_{\frac{1}{n}-}(y_1,y_2,y_3)
&= A_{\frac{1}{n}-}(\iota^n(y_1,y_2,y_3))
= A_{\frac{1}{n}-}(y_1 \1_{[0,T]}),
\end{align*}
where $y_1 \1_{[0,T]}$ denotes the constant path at $y_1$. Hence, the asserted pointwise convergence follows from the fact that $A_0 = 0$ and $A$ is right-continuous.

(ii): If $A$ is a strictly interior averaging process, then the last expression in \eqref{eqn:lem:average:weak approximation:pf:difference} is identically zero for $n$ large enough and setting $P=P_n$ gives the desired result.
\end{proof}

\begin{remark}
\label{rem:implementability}
Part~(i) of Lemma~\ref{lem:average:weak approximation} remains true if we restrict ourselves to martingale measures with almost surely continuous paths. The analog of part~(ii) for continuous martingales requires the additional assumption that there exists $t < T$ such that $A_t \equiv 1$.

The main ingredient for this assertion is \cite[Theorem~11]{Chacon1977}: for every discrete time-martingale $\{Y_n\}_{n \geq 0}$, there is a continuous-time martingale $\{Z_t\}_{t \geq 0}$ with continuous sample paths such that the processes $\{Y_n\}_{n \geq 0}$ and $\{Z_n\}_{n \geq 0}$ have the same (joint) distribution.
\end{remark}

\begin{proof}[Proof of Proposition~\ref{prop:reduction:primal}]
Let $\mu\leq_c\theta\leq_c\nu$. Assume first that condition (ii) holds and let $A$ be a strictly interior averaging process. Then by Lemma~\ref{lem:average:weak approximation}~(ii), there is $P\in\cM(\mu,\nu)$ such that $\cL^P(\int_{[0,T]}X_t\dd A_t) = \theta$. Hence,
\begin{equation*}
\theta(f)
= \EX[P]{f\Big(\int_{[0,T]}X\dd A\Big)}
\leq \bfS_{\mu,\nu}(f,\cA).
\end{equation*}
As $\theta$ was arbitrary, the claim follows.

Next, assume instead that condition (i) holds and let $A$ be an interior averaging process and $\varphi$ as in condition (i). By Lemma~\ref{lem:average:weak approximation}~(i), there is a sequence $(P_n)_{n\in\NN}\subset\cM(\mu,\nu)$ such that  $\theta_n:=\cL^{P_n}(\int_{[0,T]}X_t\dd A_t) \to \theta$ weakly. Define $f_k = f \vee (-k)$, $k \geq 1$. Then $f_k$ is bounded from below and lower semicontinuous, so $\liminf_{n\to\infty} \theta_n(f_k) \geq \theta(f_k)$ by the Portmanteau theorem.

Fix $\varepsilon > 0$. Choose first $k$ large enough such that $\nu((\varphi+k)^-) \leq \frac{\varepsilon}{2}$ and then $N$ large enough such that $\theta_n(f_k)-\theta(f_k) \geq -\frac{\varepsilon}{2}$ for all $n \geq N$. Using that $0 \leq f_k-f \leq (\varphi + k)^-$ and that $(\varphi + k)^-$ is convex, we obtain for $n \geq N$,
\begin{align*}
\theta_n(f) - \theta(f)
&= \theta_n(f-f_k) + \left(\theta_n(f_k) - \theta(f_k)\right) + \theta(f_k - f)\\
&\geq -\theta_n((\varphi+k)^-)-\frac{\varepsilon}{2}
\geq -\nu((\varphi+k)^-) - \frac{\varepsilon}{2}
\geq -\varepsilon.
\end{align*}
Thus, $\liminf_{n\to\infty} \theta_n(f) \geq \theta(f)$. Now the claim follows from
\begin{align*}
\theta(f)
&\leq \liminf_{n\to\infty} \theta_n(f)
= \liminf_{n\to\infty} \EX[P_n]{f\Big(\int_{[0,T]}X_t\dd A_t\Big)}
\leq \bfS_{\mu,\nu}(f,\cA).\qedhere
\end{align*}
\end{proof}

\subsection{Auxiliary dual problem}
\label{sec:dual problem}

Consider the \emph{auxiliary dual problem}
\begin{equation}
\label{eqn:auxiliary dual problem}
\wt\bfI_{\mu,\nu}(f)
= \inf_{(\varphi,\psi)\in\wt\cD_{\mu,\nu}(f)} \lbrace \mu(\varphi) + \nu(\psi) \rbrace,
\end{equation}
where $\wt\cD_{\mu,\nu}(f)$ denotes the set of $(\varphi,\psi) \in L^c(\mu,\nu)$ with concave $\varphi:J\to\ol\RR$ and convex $\psi:J\to\ol\RR$ such that $\varphi + \psi \geq f$ on $J$.

The dual value $\wt\bfI_{\mu,\nu}(f)$ is an upper bound for the robust superhedging price \eqref{eqn:hedging problem}:

\begin{proposition}
\label{prop:reduction:dual}
Let $f:\RR\to[0,\infty]$ be Borel. Then $\bfI_{\mu,\nu}(f,\cA) \leq \wt\bfI_{\mu,\nu}(f)$.
\end{proposition}

Proposition~\ref{prop:reduction:dual} follows immediately from the next result (Proposition~\ref{prop:dynamic part}) which shows that every $(\varphi,\psi)\in\wt\cD_{\mu,\nu}(f)$ gives rise to a semi-static superhedge for $f$ and $\cA$. More precisely, the semi-static superhedge is of the form $(\varphi,\psi,H)$ and the dynamic part $H$ can be explicitly written in terms of the ``derivatives'' of $\varphi$ and $\psi$.

Given a convex function $\psi: J \to \RR$, a Borel function $\psi':I \to \RR$ is called a \emph{subderivative} of $\psi$ if for every $x_0 \in I$, $\psi'(x_0)$ belongs to the subdifferential of $\psi$ at $x_0$, i.e.,
\begin{equation*}
\psi(x) - \psi(x_0)
\geq \psi'(x_0)(x-x_0),\quad x \in J.
\end{equation*}
Symmetrically, for a concave function $\varphi: J\to\RR$, a Borel function $\varphi':I \to \RR$ is called a \emph{superderivative} of $\varphi$ if $-\varphi'$ is a subderivative of $-\varphi$.

\begin{remark}
\label{rem:phi psi finite}
If $(\varphi,\psi)\in\wt\cD_{\mu,\nu}(f)$ and $f > -\infty$ on $J$, then $\varphi$ and $\psi$ are both finite (so that sub- and superderivatives are well defined). Indeed, we already know from Remark~\ref{rem:psi finite} that $\psi$ is finite on $J$. Moreover, $\varphi < \infty$ on $J$ by Lemma~\ref{lem:Lc:finite}~(i) and if $f > -\infty$ on $J$, then $\varphi \geq f-\psi > -\infty$, so that also $\varphi$ is finite on $J$.
\end{remark}

\begin{proposition}
\label{prop:dynamic part}
Let $f:\RR\to[0,\infty]$ be Borel and let $(\varphi,\psi)\in\wt\cD_{\mu,\nu}(f)$. Denoting the canonical process on $\hat\Omega$ by $(X,A)$, define the $\hat\FF$-adapted process $h = (h_t)_{t\in[0,T]}$ (on $\hat\Omega$) by
\begin{align}
\label{eqn:prop:dynamic part:h}
\begin{split}
h_0
&= \varphi'(X_0)(1-A_0) - \psi'(X_0)A_0,\\
h_t
&= -\varphi'(X_0) - \psi'(X_t),\quad t \in (0,T],
\end{split}
\end{align}
where $\varphi'$ is any superderivative of $\varphi$, $\psi'$ is any subderivative of $\psi$, and we set $\varphi'=\psi' = 0$ on $\RR\setminus I$.
Set $H = (h,A)$. Then $(\varphi,\psi, H) \in \cD_{\mu,\nu}(f,\cA)$ for any nonempty set $\cA$ of averaging processes.
\end{proposition}

The proof of Proposition~\ref{prop:dynamic part}, given at the end of this section, relies on the following two technical lemmas. The definition of $\Omega_{\mu,\nu}$ in \eqref{eqn:relevant paths} is crucial for the first one. We recall that (real-valued) c\`adl\`ag functions are bounded on compact intervals.

\begin{lemma}
\label{lem:dynamic part:bounded}
Let $\psi$ and $\psi'$ be as in Proposition~\ref{prop:dynamic part}. For each $\omega\in\Omega_{\mu,\nu}$, the function $[0,T] \ni t \mapsto (\omega_T-\omega_t)\psi'(\omega_t)$ is bounded.
\end{lemma}

\begin{proof}
Fix $\omega \in \Omega_{\mu,\nu}$ and write $I= (l,r)$ with $l,r \in \ol\RR$. We consider three cases: (i) $J = I$, (ii) $J=[l,r)$, and (iii) $J = [l,r]$. The case $(l,r]$ is symmetric to (ii).

(i): Suppose that $J = I = (l,r)$. We claim that $\omega$ evolves in a compact (and hence strict) subset of $I$. Suppose for the sake of contradiction hat $\inf_{t\in[0,T]} \omega_t = l \in [-\infty,\infty)$. Then there is a sequence $(t_n)_{n\in\NN} \subset [0,T]$ such that $\lim_{n\to\infty} \omega_{t_n} = l$. Passing to a subsequence if necessary, this sequence may be chosen to be either (strictly) increasing or nonincreasing to a limit $t^\star\in [0,T]$. Then, as $\omega$ is c\`adl\`ag, $\omega_{t^\star-} = l$ or $\omega_{t^\star} = l$. But then $\omega_{t^\star}  =l$ in any case by the definition of $\Omega_{\mu,\nu}$, a contradiction to $\omega_{t^\star} \in J = I$. Thus, $\inf_{t\in[0,T]} \omega_t > l$ and symmetrically $\sup_{t\in[0,T]} \omega_t < r$. This proves the claim. It follows that $(\omega_T-\omega_t)\psi'(\omega_t)$ is bounded over $t\in[0,T]$ because the subderivative $\psi'$ is bounded on compact subsets of $I$.

(ii): Suppose that $J = [l,r)$, i.e., $\nu$ has an atom in $l > -\infty$. If $\omega$ evolves in $I$, then we can argue as in (i). We may thus assume that $t^\star := \inf \lbrace t\in [0,T] : \omega_t = l \rbrace \in (0,T]$. Then, as $\omega$ is c\`adl\`ag and by the definition of $\Omega_{\mu,\nu}$, we have $\omega_u = l$ for all $u \in [t^\star,T]$). In particular, $\omega_T = l$ and it is enough to show that $[0,t^\star)\ni t \mapsto (\omega_T - \omega_t)\psi'(\omega_t)$ is bounded.

We can argue similarly as in (i) that $r':=\sup_{t\in[0,T]} \omega_t < r$, so that the path $\omega$ evolves in the compact interval $[l,r']$. Because $\psi$ is convex and finite on $(l,r)$, $\psi'$ is bounded from above on $[l,r']$. It follows that $t\mapsto(\omega_T-\omega_t)\psi'(\omega_t)$ is bounded from below on $[0,t^\star)$. To show that this function is also bounded from above, we observe that by the convexity of $\psi$,
\begin{equation}
\label{eqn:lem:dynamic part:bounded:pf:ii}
(\omega_T-\omega_t)\psi'(\omega_t)
\leq \psi(\omega_T) - \psi(\omega_t)
=\psi(l) - \psi(\omega_t).
\end{equation}
Now $\psi(l)$ is finite because $\nu$ has an atom at $l$, and $\psi$ is bounded from below on $[l,r']$ because it is finite and convex on the compact interval $[l,r']$. Using this in \eqref{eqn:lem:dynamic part:bounded:pf:ii} shows the assertion.

(iii): Suppose that $J = [l,r]$, i.e., $\nu$ has atoms at $l > -\infty$ and $r<\infty$. As in (ii), we may assume that $\omega$ hits one of the endpoints of $J$ before $T$. By symmetry, we may assume that $\omega$ hits $l$. By definition of $\Omega_{\mu,\nu}$, the path $\omega$ is then bounded away from the right endpoint $r$ (otherwise it would be captured in $r$), i.e., $\sup_{t\in[0,T]} \omega_t < r$. Now the same argument as in (ii) proves the assertion.
\end{proof}

The second technical lemma is an adaptation of \cite[Remark~4.10]{BeiglbockNutzTouzi2017} to our setting. It is used to show the admissibility condition \eqref{eqn:def:superhedge:admissibility condition} of the semi-static trading strategy in Proposition~\ref{prop:dynamic part}.

\begin{lemma}
\label{lem:expectation}
Let $(\varphi,\psi) \in L^c(\mu,\nu)$ and let $g_0,g_1:J \to \RR$ be Borel. Let $\tau$ be a $[0,T]$-valued $\FF$-stopping time such that
\begin{equation}
\label{eqn:lem:expectation:integrand}
\varphi(X_0) + \psi(X_T) + g_0(X_0)(X_\tau - X_0) + g_1(X_\tau)(X_T - X_\tau)
\end{equation}
is bounded from below on $\Omega_{\mu,\nu}$. Then for all $P \in \cM(\mu,\nu)$,
\begin{equation*}
\EX[P]{\varphi(X_0) + \psi(X_T) + g_0(X_0)(X_\tau - X_0) + g_1(X_\tau)(X_T - X_\tau)}
= \mu(\varphi)+\nu(\psi).
\end{equation*}
\end{lemma}

\begin{proof}
Let $\chi$ be a concave moderator for $(\varphi,\psi)$ with respect to $\mu \leq_c \nu$ and let $\theta$ be the law of $X_\tau$. By optional stopping, $\mu \leq_c \theta \leq_c \nu$. We expand \eqref{eqn:lem:expectation:integrand} to
\begin{align}
\label{eqn:lem:expectation:pf:decomposition}
(\varphi - \chi)(X_0) &+ (\psi + \chi)(X_T)
+ [\chi(X_0) - \chi(X_T) \\ &+ g_0(X_0)(X_\tau - X_0) + g_1(X_\tau)(X_T - X_\tau)], \nonumber
\end{align}
and observe that the first two terms are $P$-integrable. Then the assumed lower bound yields that the last term has a $P$-integrable negative part. We can therefore apply Fubini's theorem and evaluate its integral iteratively. To this end, let $Q$ be the law of $(X_0,X_\tau,X_T)$ on the canonical space $\RR^3$ with a disintegration 
\begin{align*}
\diff Q
&= \mu(\diff x_0) \otimes \kappa_0(x_0,\diff x_1) \otimes \kappa_1(x_0,x_1,\diff x_2)
\end{align*}
for martingale kernels $\kappa_0$ and $\kappa_1$. In view of the definition of $\mu(\varphi) + \nu(\psi)$ in \eqref{eqn:def:BNT integral}, we have to show that the $P$-expectation of the last term in \eqref{eqn:lem:expectation:pf:decomposition} is $(\mu-\nu)(\chi)$.

To this end, we observe that for $\mu\otimes\kappa_0$-a.e.~$(x_0,x_1) \in J^2$,
\begin{align}
\label{eqn:lem:expectation:pf:10}
\int_J &\left[ \chi(x_0) - \chi(x_2) + g_0(x_0)(x_1-x_0) + g_1(x_1)(x_2-x_1) \right]\,\kappa_1(x_0,x_1,\diff x_2) \nonumber \\
&\;= \int_J \left[\chi(x_0) - \chi(x_2) + g_0(x_0)(x_1-x_0) \right]\,\kappa_1(x_0,x_1,\diff x_2)\\
&\;= \chi(x_0) - \int_J \chi(x_2) \kappa_1(x_0,x_1,\diff x_2) + g_0(x_0)(x_1-x_0). \nonumber
\end{align}
Integrating the left-hand side of \eqref{eqn:lem:expectation:pf:10} against $\mu\otimes\kappa_0$ gives the $P$-expectation of the last term in \eqref{eqn:lem:expectation:pf:decomposition}. It thus remains to show that the corresponding integral of the right-hand side equals $(\mu-\nu)(\chi)$. Integrating the right-hand side of \eqref{eqn:lem:expectation:pf:10} first against $\kappa_0(x_0,\diff x_1)$ yields for $\mu$-a.e.~$x_0\in J$,
\begin{align}
\label{eqn:lem:expectation:pf:20}
\chi(x_0) - \int_J\chi(x_2) \,\kappa(x_0,\diff x_2),
\end{align}
where $\kappa(x_0,\cdot) = \int_J \kappa_1(x_1,\cdot)\,\kappa_0(x_0,\diff x_1)$ is again a martingale kernel. Finally, the integral of \eqref{eqn:lem:expectation:pf:20} against $\mu$ is
\begin{align*}
\int_J \left[ \chi(x_0) - \int_J\chi(x_2) \,\kappa(x_0,\diff x_2) \right] \mu(\diff x_0).
\end{align*}
Noting that $\mu\otimes\kappa$ is a disintegration of a one-step martingale measure on $\RR^2$ with marginals $\mu$ and $\nu$, the last term equals $(\mu-\nu)(\chi)$ by Lemma~\ref{lem:concave integral:disintegration}.
\end{proof}

\begin{proof}[Proof of Proposition~\ref{prop:dynamic part}]
First, we show that $(\varphi,\psi,H)$ is a semi-static trading strategy. As $h$ and $A$ are clearly $\hat\FF$-adapted and $(\varphi,\psi)\in L^c(\mu,\nu)$ by assumption, it remains to check condition \eqref{eqn:def:strategy:pathwise integrability} (with $Y^A$ replaced by $A$). So fix an averaging process $A$ and note that $H^A = (h^A,A)$. The only nontrivial part in proving $H^A \in L(\Omega_{\mu,\nu})$ is to show that $(X_T-X_t)h^A_t$ is $\diff A$-integrable on $(0,T]$ for each path in $\Omega_{\mu,\nu}$. To this end, note that $\varphi'(X_0)$ and $\psi'(X_0)$ are finite because $X_0 \in I$. It thus suffices to show that $(X_T-X_t)\psi'(X_t)$ is bounded on $[0,T]$ for each path in $\Omega_{\mu,\nu}$; this is the content of Lemma~\ref{lem:dynamic part:bounded}.

Second, we show the superhedging property \eqref{eqn:def:superhedge}. Fix an averaging process $A$ and a path in $\Omega_{\mu,\nu}$. To ease the notation, we write $h$ instead of $h^A$ in the following. Note, however, that $h^A$ has the same formal expression as $h$ in \eqref{eqn:prop:dynamic part:h}, but with $A$ being the fixed averaging process (and not the second component of the canonical process on $\hat\Omega$).

Using the definitions of $H\diamond X_T$ and $h$ as well as the fact that $A_0 = \Delta A_0$, we obtain
\begin{align*}
H \diamond X_T
&= (X_T-X_0)h_0 + \int_{(0,T]}(X_T-X_t) h_t \dd A_t\\
&= (X_T-X_0)\varphi'(X_0) - \int_{[0,T]} (X_T-X_t) (\varphi'(X_0) + \psi'(X_t)) \dd A_t.
\end{align*}
Then, using that $\diff A$ is a probability measure on $[0,T]$, the concavity of $\varphi$ and the convexity of $\psi$, and Jensen's inequality, we can estimate
\begin{align}
H \diamond X_T
&= \int_{[0,T]} \varphi'(X_0)(X_t - X_0) \dd A_t - \int_{[0,T]} \psi'(X_t)(X_T - X_t) \dd A_t \label{eqn:prop:dynamic part:pf:calculation}\\
& \geq \varphi'(X_0) \left(\int_{[0,T]} X_t \dd A_t - X_0\right) - \int_{[0,T]}\left( \psi(X_T) - \psi(X_t)\right) \dd A_t \notag\\
& \geq \varphi\left(\int_{[0,T]} X_t \dd A_t\right) - \varphi(X_0) - \psi(X_T) + \int_{[0,T]} \psi(X_t) \dd A_t \notag\\
& \geq \varphi\left(\int_{[0,T]} X_t \dd A_t\right) - \varphi(X_0) - \psi(X_T) + \psi\left(\int_{[0,T]} X_t \dd A_t\right).\notag
\end{align}
Rearranging terms and using that $\varphi + \psi \geq f$ on $J$, we find
\begin{align*}
\varphi(X_0) + \psi(X_T) + H \diamond X_T \geq f\left(\int_{[0,T]} X_t \dd A_t\right).
\end{align*}

Third, we show the admissibility condition \eqref{eqn:def:superhedge:admissibility condition}. Fix an averaging process $A$ and $P \in \cM(\mu,\nu)$. Define the family of $\FF$-stopping times $C_s$, $s\in(0,1)$, by
\begin{align*}
C_s
&= \inf\lbrace t \in [0,T] : A_t > s \rbrace
\end{align*}
and note that $0 \leq C_s \leq T$ for $s\in(0,1)$ because $A_T =1$. Then using the family $C_s$ as a time change (cf.~\cite[Proposition~0.4.9]{RevuzYor1999}) for the integral in \eqref{eqn:prop:dynamic part:pf:calculation} yields
\begin{align} \label{eqn:prop:dynamic part:pf:integrand}
&\varphi(X_0) + \psi(X_T) + H\diamond X_T \\
&\;=\int_0^1 \lbrace \varphi(X_0) + \psi(X_T) +\varphi'(X_0) (X_{C_s} - X_0) - \psi'(X_{C_s})(X_T-X_{C_s})\rbrace \dd s. \nonumber
\end{align}
Now, suppose that the integrand in \eqref{eqn:prop:dynamic part:pf:integrand} is bounded from below, uniformly over $s\in(0,1)$ and $\omega\in\Omega_{\mu,\nu}$. Then by Lemma~\ref{lem:expectation}, the $P$-expectation of the integrand equals $\mu(\varphi) + \nu(\psi)$ for each $s\in(0,1)$. Using this together with Tonelli's theorem and \eqref{eqn:prop:dynamic part:pf:integrand} gives
\begin{align*}
\EX[P]{\varphi(X_0) + \psi(X_T) + H \diamond X_T}
&= \mu(\varphi) + \nu(\psi),
\end{align*}
so that \eqref{eqn:def:superhedge:admissibility condition} holds.

It remains to show that the integrand in \eqref{eqn:prop:dynamic part:pf:integrand} is uniformly bounded from below. This follows from concavity of $\varphi$ and convexity of $\psi$ together with the fact that $\varphi + \psi \geq f \geq 0$ on $J$:
\begin{align*}
\varphi(X_0) &+ \psi(X_T) +\varphi'(X_0) (X_t - X_0) - \psi'(X_t)(X_T-X_t)\\
&\;\geq \varphi(X_t) + \psi(X_t)
\geq f(X_t)
\geq 0, \quad t\in[0,T].
\end{align*}
This completes the proof.
\end{proof}

\subsection{Duality}
\label{sec:auxiliary duality}

We now turn to the duality between the auxiliary problems $\wt\bfS_{\mu,\nu}(f)$ and $\wt\bfI_{\mu,\nu}(f)$.

\begin{theorem}
\label{thm:StrongAuxiliaryDuality}
Let $\mu \leq_c \nu$ be irreducible with domain $(I,J)$ and let $f: \RR \to [0,\infty]$.
\begin{enumerate}
\item If $f$ is upper semianalytic, then $\wt\bfS_{\mu,\nu}(f) = \wt\bfI_{\mu,\nu}(f) \in [0,\infty]$.
\item If $\wt\bfI_{\mu,\nu}(f) < \infty$, then there exists a dual minimizer $(\varphi,\psi) \in \wt\cD_{\mu,\nu}(f)$.
\end{enumerate}
\end{theorem}

A couple of remarks are in order.

\begin{remark}
We only state the duality for one irreducible component. One can formulate and prove the full duality for arbitrary marginals $\mu \leq_c \nu$ in analogy to \cite[Section~7]{BeiglbockNutzTouzi2017}. We omit the details in the interest of brevity.
\end{remark}

\begin{remark}
\label{rem:auxiliary duality:relax lower bound}
The lower bound on $f$ in Theorem~\ref{thm:StrongAuxiliaryDuality} can be relaxed. Indeed, suppose that $f:\RR\to\ol\RR$ is upper semianalytic and bounded from below by an affine function $g$.

We first consider the primal problem. Because $g$ is affine and any $\mu\leq_c\theta\leq_c\nu$ has the same mass and barycenter as $\mu$,
\begin{align*}
\theta(f-g)
&= \theta(f) - \theta(g)
= \theta(f) - \mu(g).
\end{align*}
Thus,
\begin{align}
\label{eqn:rem:auxiliary duality:extension:S}
\wt\bfS_{\mu,\nu}(f-g)
&= \wt\bfS_{\mu,\nu}(f) - \mu(g).
\end{align}

Regarding the dual problem, we note that $(\varphi,\psi)\in\wt\cD_{\mu,\nu}(f-g)$ if and only if $(\varphi+g,\psi)\in\wt\cD_{\mu,\nu}(f)$ and that by Lemma~\ref{lem:Lc:finite}~(iii),
\begin{align*}
\mu(\varphi) + \nu(\psi)
&= \lbrace \mu(\varphi+g) + \nu(\psi) \rbrace - \mu(g).
\end{align*}
Hence,
\begin{align}
\label{eqn:rem:auxiliary duality:extension:I}
\wt\bfI_{\mu,\nu}(f-g)
&= \wt\bfI_{\mu,\nu}(f) - \mu(g).
\end{align}
Because $f-g$ is nonnegative, the left-hand sides of \eqref{eqn:rem:auxiliary duality:extension:S}--\eqref{eqn:rem:auxiliary duality:extension:I} coincide by Theorem~\ref{thm:StrongAuxiliaryDuality}~(i). Therefore, $\wt\bfS_{\mu,\nu}(f) = \wt\bfI_{\mu,\nu}(f) \in (-\infty,\infty]$.

Moreover, if $\wt\bfI_{\mu,\nu}(f) < \infty$, then also $\wt\bfI_{\mu,\nu}(f-g) < \infty$ and a dual minimizer $(\varphi,\psi)\in\wt\cD_{\mu,\nu}(f-g)$ for $\wt\bfI_{\mu,\nu}(f-g)$ exists by Theorem~\ref{thm:StrongAuxiliaryDuality}~(ii). Now the above shows that $(\varphi+g,\psi) \in \wt\cD_{\mu,\nu}(f)$ is a dual minimizer for $\wt\bfI_{\mu,\nu}(f)$.
\end{remark}

The proof of Theorem~\ref{thm:StrongAuxiliaryDuality} is based on several preparatory results. We start with the crucial closedness property of the dual space in the spirit of \cite[Proposition~5.2]{BeiglbockNutzTouzi2017}.

\begin{proposition}
\label{prop:DualClosedness}
Let $\mu \leq_c \nu$ be irreducible with domain $(I,J)$, let $f,f_n: J \to [0,\infty]$ be such that $f_n \to f$ pointwise, and let $(\varphi_n,\psi_n) \in \wt\cD_{\mu,\nu}(f_n)$ with $\sup_n \lbrace\mu(\varphi_n) + \nu(\psi_n)\rbrace < \infty$. Then there is $(\varphi,\psi) \in \wt\cD_{\mu,\nu}(f)$ such that $\mu(\varphi) + \nu(\psi) \leq \liminf_{n \to \infty} \lbrace \mu(\varphi_n) + \nu(\psi_n) \rbrace$.
\end{proposition}

\begin{proof}
Let $h_n = \varphi_n':I\to\RR$ be a superderivative of the concave function $\varphi_n$. As
\begin{align*}
\varphi_n(x) + \psi_n(y) + h_n(x)(y-x)
&\geq \varphi_n(y) + \psi_n(y)
\geq f_n(y) \geq 0, \quad (x,y) \in I\times J,
\end{align*}
$(\varphi_n,\psi_n,h_n)$ is in the dual space $\cD^c_{\mu,\nu}(0)$ of \cite{BeiglbockNutzTouzi2017}. Hence, following the line of reasoning in the proof of \cite[Proposition~5.2]{BeiglbockNutzTouzi2017} (which is based on Komlos' lemma; we recall that convex combinations of convex (concave) functions are again convex (concave)), we may assume without loss of generality that 
\begin{align*}
\varphi_n \to \bar\varphi\quad\mu\text{-a.e.}
\quad\text{and}\quad
\psi_n \to \bar\psi\quad\nu\text{-a.e.}
\end{align*}
for some $(\bar\varphi,\bar\psi)\in L^c(\mu,\nu)$. Moreover, the arguments in \cite{BeiglbockNutzTouzi2017} also show that $\mu(\bar\varphi) + \nu(\bar\psi) \leq \liminf_{n \to \infty} \lbrace \mu(\varphi_n) + \nu(\psi_n) \rbrace$.

Now, define the functions $\varphi,\psi:J\to\ol\RR$ by $\varphi := \liminf_{n\to\infty} \varphi_n$ and $\psi := \limsup_{n\to\infty} \psi_n$. Then $\varphi$ is convex, $\psi$ is concave, $\varphi = \bar\varphi$ $\mu$-a.e., and $\psi = \bar\psi$ $\nu$-a.e. In particular, $(\varphi,\psi) \in L^c(\mu,\nu)$ and $\mu(\varphi) + \nu(\psi) \leq \liminf_{n \to \infty} \lbrace \mu(\varphi_n) + \nu(\psi_n) \rbrace$. Furthermore, as $\varphi_k + \psi_k \geq f_k$ on $J$, we have for each $n$ that
\begin{align*}
\inf_{k \geq n} \varphi_k + \sup_{k \geq n} \psi_k
&\geq \inf_{k \geq n} (\varphi_k + \psi_k)
\geq \inf_{k \geq n} f_k\quad\text{on } J.
\end{align*}
Sending $n\to\infty$ gives $\varphi + \psi \geq f$. In summary, $(\varphi,\psi) \in \wt\cD_{\mu,\nu}(f)$.
\end{proof}

We proceed to show strong duality for bounded upper semicontinuous functions. 

\begin{lemma}
\label{lem:ContinuousDuality}
Let $f: \RR \to [0,\infty]$ be bounded and upper semicontinuous. Then $\wt\bfS_{\mu,\nu}(f) = \wt\bfI_{\mu,\nu}(f)$.
\end{lemma}

The proof is based on a Hahn--Banach separation argument similar to \cite[Lemma~6.4]{BeiglbockNutzTouzi2017}.

\begin{proof}
We first show the weak duality inequality. Let $\mu\leq_c\theta\leq_c\nu$ and $(\varphi,\psi)\in\wt\cD_{\mu,\nu}(f)$. In particular, $\varphi+\psi$ is bounded from below. Then by Lemma~\ref{lem:ConvexIntegrationRules}~(iii)--(iv), 
\begin{align}
\label{eqn:lem:ContinuousDuality:pf:weak duality}
\theta(f)
&\leq \theta(\varphi+\psi)
= \theta(\varphi) + \theta(\psi)
\leq \theta(\varphi) + \nu(\psi)
\leq \mu(\varphi) + \nu(\psi),
\end{align}
and the inequality $\wt\bfS_{\mu,\nu}(f) \leq \wt\bfI_{\mu,\nu}(f)$ follows.

The converse inequality is based on a Hahn--Banach argument, so let us introduce a suitable space. By the de la Vall\'ee--Poussin theorem, there is an increasing convex function $\zeta_\nu:\RR_+\to\RR_+$ of superlinear growth such that $x \mapsto \zeta_\nu(|x|)$ is $\nu$-integrable. Now, set $\zeta(x) = 1+\zeta_\nu(|x|)$, $x \in \RR$, and denote by $C_\zeta$ the space of all continuous functions $f:\RR\to\RR$ such that $f/\zeta$ vanishes at infinity. We endow $C_\zeta$ with the norm $\|f\|_\zeta := \|f/\zeta\|_\infty$. With this notation, the same arguments as in the proof of \cite[Lemma~6.4]{BeiglbockNutzTouzi2017} show that the dual space $C^*_\zeta$ of continuous linear functionals on $C_\zeta$ can be represented by finite signed measures.

Fix $f \in C_\zeta$. Then
\begin{align}
\label{eqn:lem:ContinuousDuality:pf:f}
-\zeta(x)\| f \|_\zeta 
&\leq f(x)
\leq \zeta(x)\| f \|_\zeta,\quad x \in J.
\end{align}
Because $\zeta_\nu$ is convex and $x \mapsto \zeta_\nu(|x|)$ is $\nu$-integrable, we have $\theta(\zeta) \leq \nu(\zeta) < \infty$ for all $\mu\leq_c\theta\leq_c\nu$. This together with \eqref{eqn:lem:ContinuousDuality:pf:f} shows that $\wt\bfS_{\mu,\nu}(f)$ is finite. Thus, adding a suitable constant to $f$, we may assume that $\wt\bfS_{\mu,\nu}(f) = 0$. For the following Hahn--Banach argument, we consider the convex cone
\begin{align*}
K
&:= \{g \in C_\zeta : \wt\bfI_{\mu,\nu}(g) \leq 0\}.
\end{align*}
Proposition~\ref{prop:DualClosedness} implies that $K$ is closed.

Suppose for the sake of contradiction that $\wt\bfI_{\mu,\nu}(f) > 0$. Then, by the Hahn--Banach theorem, $K$ and $f$ can be strictly separated by a continuous linear functional on $C_\zeta$. That is, there is a finite signed measure $\rho$ such that $\rho(f) > 0$ and $\rho(g) \leq 0$ for all $g\in K$. For any compactly supported nonnegative continuous function $g\in C_\zeta$, we have $\wt\bfI_{\mu,\nu}(-g) \leq 0$. That is, $-g \in K$ and hence $\rho(-g) \leq 0$. This shows that $\rho$ is a (nonnegative) finite measure. Multiplying $\rho$ by a positive constant if necessary, we may assume that $\rho$ has the same mass as $\mu$ and $\nu$. Next, let $\psi$ be convex and of linear growth. Then $\psi \nu(\RR) - \nu(\psi)\in K$ and $-\psi \mu(\RR) + \mu(\psi) \in K$. Using that $\rho\leq0$ for these two functions yields $\mu(\psi) \leq \rho(\psi) \leq \nu(\psi)$. We conclude that $\mu\leq_c\rho\leq_c\nu$. But now $\rho(f) > 0$ contradicts $\wt\bfS_{\mu,\nu}(f) = 0$. Thus, $\wt\bfI_{\mu,\nu}(f) \leq \wt\bfS_{\mu,\nu}(f)$.

Finally, let $f$ be bounded and upper semicontinuous and choose $f_n \in C_b(\RR) \subseteq C_\zeta$ such that $f_n \searrow f$. By the above, we have $\wt\bfS_{\mu,\nu}(f_n) = \wt\bfI_{\mu,\nu}(f_n)$ for all $n$. We show below that $\lim_{n\to\infty}\wt\bfS_{\mu,\nu}(f_n) = \wt\bfS_{\mu,\nu}(f)$. Using this and the monotonicity of $\wt\bfI_{\mu,\nu}$, we obtain
\begin{align*}
\wt\bfI_{\mu,\nu}(f)
&\leq \lim_{n\to\infty} \wt\bfI_{\mu,\nu}(f_n)
= \lim_{n\to\infty} \wt\bfS_{\mu,\nu}(f_n)
= \wt\bfS_{\mu,\nu}(f)
\leq \wt\bfI_{\mu,\nu}(f).
\end{align*}
So strong duality holds for bounded upper semicontinuous functions.

It remains to argue that $\lim_{n\to\infty}\wt\bfS_{\mu,\nu}(f_n) = \wt\bfS_{\mu,\nu}(f)$. We show more generally that $\wt\bfS_{\mu,\nu}$ is continuous along decreasing sequences of bounded upper semicontinuous functions. So let $f_n \searrow f$ be a convergent sequence of bounded upper semicontinuous functions. Fix $\varepsilon > 0$ and set $\ell := \lim_{n\to\infty} \wt\bfS_{\mu,\nu}(f_n)$. Then for each $n$, $\ell\leq \bfS_{\mu,\nu}(f_n) < \infty$ and thus the set
\begin{align*}
A_n
&:= \{\mu \leq_c \theta \leq_c \nu : \theta(f_n) \geq \ell - \varepsilon\}
\end{align*}
is nonempty. Moreover, each $A_n$ is a closed subset of the weakly compact set $\lbrace \theta : \mu \leq_c \theta \leq_c \nu \rbrace$ and $A_{n+1} \subseteq A_n$. Therefore, there exists a $\theta'$ in the intersection $\cap_{n \geq 1} A_n$. We then obtain by monotone convergence that
\begin{align*}
\wt\bfS_{\mu,\nu}(f)
&\geq \theta'(f)
= \lim_{n\to\infty} \theta'(f_n) \geq \ell - \varepsilon.
\end{align*}
This implies that $\wt\bfS_{\mu,\nu}(f) \geq \ell$ as $\varepsilon$ was arbitrary. The converse inequality follows from the monotonicity of $\wt\bfS_{\mu,\nu}$. This completes the proof.
\end{proof}

\begin{proof}[Proof of Theorem \ref{thm:StrongAuxiliaryDuality}]
(i): This is a consequence of Lemma~\ref{lem:ContinuousDuality} and a capacitability argument that is almost verbatim to \cite[Section~6]{BeiglbockNutzTouzi2017}. The same arguments can be found in \cite{Kellerer84}. We therefore omit these elaborations.

(ii): Applying Proposition~\ref{prop:DualClosedness} to the constant sequence $f_n = f$ and a minimizing sequence $(\varphi_n,\psi_n) \in \wt\cD_{\mu,\nu}(f)$ of $\wt\bfI_{\mu,\nu}(f)$ yields a dual minimizer.
\end{proof}

We are now in a position to prove the duality between the robust pricing and superhedging problems.

\begin{proof}[Proof of Theorem \ref{thm:StrongDuality}]
By Proposition~\ref{prop:reduction:primal}, Lemma~\ref{lem:weak duality}, and Proposition~\ref{prop:reduction:dual},
\begin{align*}
\wt\bfS_{\mu,\nu}(f)
&\leq \bfS_{\mu,\nu}(f,\cA)
\leq \bfI_{\mu,\nu}(f,\cA)
\leq \wt\bfI_{\mu,\nu}(f),
\end{align*}
and Theorem~\ref{thm:StrongAuxiliaryDuality} shows that $\wt\bfS_{\mu,\nu}(f) = \wt\bfI_{\mu,\nu}(f)$. Hence,
\begin{align}
\label{eqn:thm:StrongDuality:pf:equality}
\wt\bfS_{\mu,\nu}(f)
&= \bfS_{\mu,\nu}(f,\cA)
= \bfI_{\mu,\nu}(f,\cA)
= \wt\bfI_{\mu,\nu}(f).
\end{align}
In particular, the quantities in \eqref{eqn:thm:StrongDuality:pf:equality} are all independent of the choice of $\cA$ (as long as one of the two conditions of Theorem~\ref{thm:StrongDuality} holds).

If $\bfI_{\mu,\nu}(f,\cA) < \infty$, then $\wt\bfI_{\mu,\nu}(f) < \infty$ and hence there is an optimizer $(\varphi,\psi) \in \wt\cD_{\mu,\nu}(f)$ for $\wt\bfI_{\mu,\nu}(f)$. Then Proposition~\ref{prop:dynamic part} provides an $H = (h,A)$ such that $(\varphi,\psi,H) \in \cD_{\mu,\nu}(f,\cA)$. By \eqref{eqn:thm:StrongDuality:pf:equality} and the definition of $\bfI_{\mu,\nu}(f,\cA)$, $(\varphi,\psi,H)$ is an optimizer for $\bfI_{\mu,\nu}(f,\cA)$.
\end{proof}

\begin{remark}
\label{rem:strong duality with FV strategies}
Strong duality (without dual attainment) for the robust pricing and superhedging problems continues to hold if we restrict ourselves to trading strategies whose dynamic part is of finite variation.

First, observe that the process $\hat H$ defined by \eqref{eqn:integral:integrand} is of finite variation when $h$ is bounded. Recalling the definition \eqref{eqn:prop:dynamic part:h} of $h$ in Proposition~\ref{prop:dynamic part}, we see that $h$ is bounded on $\{\omega : \omega_T \in J^\circ\}$ as these paths are bounded in a compact subset of $J$, on which $\psi'$ is bounded. This will more generally hold for almost all paths if $\psi'$ is uniformly bounded on $J$. Therefore, strong duality (and dual attainment in strategies of finite variation) holds if $J$ is open.

Second, consider the case $J = [a,b)$ for some $-\infty < a < b \leq \infty$. Suppose that the assumptions of Theorem~\ref{thm:StrongDuality} hold and that $\bfI_{\mu,\nu}(f,\cA) < \infty$, and let $(\varphi,\psi) \in \wt\cD_{\mu,\nu}(f)$ be a dual auxiliary optimizer. Then $\psi(a) < \infty$ as $\nu$ has an atom at $a$ (cf.~Lemma~\ref{lem:Lc:finite}). If $\psi'(a) > -\infty$, then the same argument as above shows that the dynamic trading strategy constructed in Proposition~\ref{prop:dynamic part} is of finite variation. If $\psi'(a) = -\infty$, then we construct a sequence of functions
\begin{align*}
\psi_k(x)
:= \begin{cases}
\psi(x) & \text{for } x \geq a+\tfrac{1}{k}, \\
\psi(a) + k(x-a) (\psi(a+\tfrac{1}{k}) - \psi(a)) & \text{for } x < a+\tfrac{1}{k}.
\end{cases}
\end{align*}
that approximates $\psi$ by linear interpolation on the interval $[a,a+\frac{1}{k}]$. We then have $\psi_k \searrow \psi$ and $\mu(\varphi) + \nu(\psi_k) \searrow \mu(\varphi) + \nu(\psi) = \wt\bfI_{\mu,\nu}(f)$ as $k\to \infty$. Since $\psi'_k(a) > -\infty$, the associated process $H^{(k)}$ is of finite variation almost surely. The cases $J = (a,b]$ and $J = [a,b]$ are analogous.
\end{remark}

\subsection{Structure of primal and dual optimizers}
\label{sec:structure}

If a primal optimizer to the auxiliary problem exists, we can derive some necessary properties for the dual optimizer.

\begin{proposition}
\label{prop:DualProperties}
Let $\mu \leq_c \nu$ be irreducible with domain $(I,J)$ and let $f: \RR \to [0,\infty]$ be Borel. Suppose that $\wt\bfS_{\mu,\nu}(f) = \wt\bfI_{\mu,\nu}(f)$, that $\mu\leq_c\theta\leq_c\nu$ is an optimizer for $\wt\bfS_{\mu,\nu}(f)$, and that $(\varphi,\psi)\in\wt\cD_{\mu,\nu}(f)$ is an optimizer for $\wt\bfI_{\mu,\nu}(f)$. Then
\begin{enumerate}
\item $\varphi+\psi = f$ $\theta$-a.e.,
\item $\varphi$ is affine on the connected components of $\{u_\mu < u_\theta\}$,
\item $\psi$ is affine on the connected components of $\{u_\theta < u_\nu\}$,
\item $\varphi$ does not have a jump at a finite endpoint $b$ of $J$ if $\theta(\{b\}) > 0$, and
\item $\psi$ does not have a jump at a finite endpoint $b$ of $J$ if $\theta(\{b\}) < \nu(\{b\})$.
\end{enumerate}
\end{proposition}

\begin{proof}
As in the proof of Lemma~\ref{lem:ContinuousDuality}, we obtain (cf.~\eqref{eqn:lem:ContinuousDuality:pf:weak duality}) that
\begin{align*}
\theta(f)
&\leq \theta(\varphi+\psi)
\leq \mu(\varphi) + \nu(\psi).
\end{align*}
By the absence of a duality gap as well as the optimality of $\theta$ and $(\varphi,\psi)$, all inequalities are equalities:
\begin{align}
\label{eqn:prop:DualProperties:pf:equality}
\theta(f)
&= \theta(\varphi+\psi)
= \mu(\varphi) + \nu(\psi).
\end{align}

Now (i) follows from the first equality in \eqref{eqn:prop:DualProperties:pf:equality} and the fact that $\varphi+\psi\geq f$ on $J$. Rearranging the second equality, we can write
\begin{align*}
0
&= \{\mu(\varphi) + \nu(\psi) \} - \theta(\varphi+\psi) \\
&= \{\mu(\varphi) + \nu(\psi) \} - \{\theta(\varphi) + \nu(\psi)\} + \{\theta(\varphi) + \nu(\psi)\}- \theta(\varphi+\psi).
\end{align*}
Using the definition \eqref{eqn:def:BNT integral} of the first three expressions (using $\varphi$ as a concave moderator for the first two terms and $-\psi$ for the third; cf.~Lemma~\ref{lem:ConvexIntegrationRules}~(i)), we obtain
\begin{align}
\label{eqn:prop:DualProperties:pf:concave integrals}
0
&= (\mu - \nu)(\varphi) - (\theta - \nu)(\varphi) + (\theta - \nu)(-\psi) \notag\\
&= (\mu - \theta)(\varphi)+(\theta-\nu)(-\psi),
\end{align}
where the last equality is a direct consequence of the definitions of $(\mu - \nu)(\varphi)$ and $(\theta - \nu)(\varphi)$ (cf.~\eqref{eqn:def:concave integral}). Both terms on the right-hand side of \eqref{eqn:prop:DualProperties:pf:concave integrals} are nonnegative by definition and hence must vanish:
\begin{align*}
0
&= (\mu - \theta)(\varphi)
= \int_I (u_\mu - u_\theta) \dd\varphi'' + \int_{J \setminus I} |\Delta \varphi| \dd\theta
\end{align*}
and similarly for $(\theta-\nu)(-\psi)$. This implies that $\varphi'' = 0$ on $\{u_\mu < u_\theta\}$ (which is assertion (ii)) and that $|\Delta \varphi| = 0$ for every endpoint of $J$ on which $\theta$ has an atom (which is assertion (iv)). The proofs of (iii) and (v) are similar.
\end{proof}

The next result shows that for upper semicontinuous $f$, there is a maximizer for $\wt\bfS_{\mu,\nu}(f)$ which is maximal with respect to the convex order. We omit the proof in the interest of brevity.

\begin{proposition}
\label{prop:SecondaryOptimizer}
Let $\mu \leq_c \nu$ be irreducible and let $f:\RR \to [0,\infty]$ be upper semicontinuous and bounded from above by a convex, continuous, and $\nu$-integrable function. Furthermore, fix a strictly convex function $g:\RR \to \RR$ with linear growth, and consider the ``secondary'' optimization problem
\begin{equation}
\label{eqn:prop:SecondaryOptimizer}
\sup_{\theta \in \Theta(f)} \theta(g),
\end{equation}
where $\Theta(f) := \lbrace \theta : \mu \leq_c \theta \leq_c \nu \;\text{ and }\; \theta(f) \geq \wt\bfS_{\mu,\nu}(f) \rbrace$ is the set of optimizers of the auxiliary primal problem.
\begin{enumerate}
\item $\Theta(f)$ is non-empty, convex, and weakly compact and \eqref{eqn:prop:SecondaryOptimizer} admits an optimizer.
\item Any optimizer $\theta$ of \eqref{eqn:prop:SecondaryOptimizer} has the following properties:
\begin{itemize}
\item $\theta$ is maximal in $\Theta(f)$ with respect to the convex order.
\item If $O$ is an open interval such that $O \subseteq \lbrace u_\theta < u_\nu \rbrace$ and $f\vert_O$ is convex, then $\theta(O) = 0$.
\item If $K$ is an interval such that $K^\circ \subseteq \lbrace u_\mu < u_\theta \rbrace$, $f\vert_K$ is strictly concave, and $\theta(K)>0$, then $\theta\vert_K$ is concentrated in a single atom.
\end{itemize}
\end{enumerate}
\end{proposition}

The following example shows that the set optimizers for $\wt\bfS_{\mu,\nu}(f)$ can have multiple maximal or minimal elements with respect to the convex order; there is in general no greatest or least element for this partially ordered set.

\begin{example}
Let $\mu = \delta_0$ and $\nu = \frac{1}{3}(\delta_{-1} + \delta_0 + \delta_1)$ and let $f$ be piecewise linear with $f(-1) = f(1) = 3$, $f(-1/2) = f(1/2) = 2$, and $f(0) = 0$. We claim that there is no greatest or least primal optimizer.

We construct candidate primal and dual optimizers as follows. On the primal side, set $\theta_1 = \tfrac{2}{3}\delta_{-\frac{1}{2}} + \tfrac{1}{3}\delta_1$ and $\theta_2 = \tfrac{1}{3}\delta_{-1} + \tfrac{2}{3}\delta_{\frac{1}{2}}$. On the dual side, set $\varphi \equiv 0$ and let $\psi$ be the convex function that interpolates linearly between $\psi(-1) = \psi(1) = 3$ and $\psi(0) = 1$. Direct computations yield $\theta_1(f) = \theta_2(f) = \frac{7}{3} = \nu(\psi)$ which shows that $\theta_1$ and $\theta_2$ are primal optimizers and that $(\varphi,\psi)$ is a dual optimizer.

First, we show that there is no primal optimizer which dominates both $\theta_1$ and $\theta_2$ in convex order. Indeed, one can check that $u_\nu = \max(u_{\theta_1},u_{\theta_2})$, so that $\nu$ is the only feasible primal element which dominates both $\theta_1$ and $\theta_2$ in convex order. But $\nu(f) = 2 < \frac{7}{3}$ and therefore $\nu$ is not optimal.

Second, we show that there is no primal optimizer which is dominated by both $\theta_1$ and $\theta_2$. Indeed, one can check that $\lbrace u_\mu < \min(u_{\theta_1},u_{\theta_2})\rbrace = (-\frac{1}{2},\frac{1}{2})$, so that every feasible primal element that is dominated by both $\theta_1$ and $\theta_2$ must be concentrated on $[-\frac{1}{2},\frac{1}{2}]$. But $f \leq 2$ on $[-\frac{1}{2},\frac{1}{2}]$, so that no primal optimizer can be concentrated on this interval.
\end{example}

We conclude this section with an example that shows that primal attainment does not hold in general if $f$ is not upper semicontinuous.

\begin{example}
\label{ex:no primal attainment}
Let $\mu = \delta_0$, $\nu = \frac{1}{2}(\delta_{-1}+\delta_1)$, and set $f(x) := |x|\1_{(-1,1)}(x)$. Then $\mu \leq_c \nu$ is irreducible with domain $((-1,1),[-1,1])$. Considering the sequence $\theta_n := \frac{1}{2}(\delta_{-1+\frac{1}{n}} + \delta_{1-\frac{1}{n}})$, one can see that $\wt\bfS_{\mu,\nu}(f) \geq 1$. But there is no $\mu \leq_c \theta \leq_c \nu$ such that $\theta(f) \geq 1$ because $f < 1$ on $[-1,1]$.
\end{example}

\section{Examples}
\label{sec:examples}

Two common payoff functions are risk reversals and butterfly spreads. In this section, we provide solutions to the auxiliary primal and dual problems for these payoffs. Throughout this section, we fix irreducible marginals $\mu \leq_c \nu$ and denote their common total mass and first moment by $m_0$ and $m_1$, respectively.

\subsection{Risk reversals}

The payoff function of a risk reversal is of the form
\begin{align*}
f(x)
&= -(a-x)_+ + (x-b)_+,
\end{align*}
for fixed $a<b$. The following result provides a simple geometric construction of the primal and dual optimizers in terms of the potential functions of $\mu$ and $\nu$.\footnote{The authors thank David Hobson for the idea of this construction.} We recall that any convex function $u$ lying between the potential functions $u_\mu$ and $u_\nu$ is the potential function of a measure $\theta$ which is in convex order between $\mu$ and $\nu$ (cf., e.g., \cite{ChaconWalsh1976}).

\begin{figure}
\centering
\begin{tabular}{cc}
\resizebox{0.4\textwidth}{!}{
\begin{tikzpicture}[scale=0.5]
\draw[domain=-9:9,variable=\x, very thick] plot ({\x},{sqrt(27+\x*\x)});
\draw[domain=-9.5:9.5,variable=\x, very thick] plot ({\x},{sqrt(8.33333+\x*\x)});
\draw[domain=-10:10,variable=\x, dashed] plot ({\x},{abs(\x)});
\draw[domain=-2:3,variable=\x] plot ({\x},{6 + 0.5*(\x-3)});
\draw[->,>=stealth] (-10,0) -- (10,0) {};

\node at (-5,7.7) {$u_\nu$};
\node at (-5.9,7) {$u_\mu$};

\draw[->,>=stealth] (-10,-9) -> (10,-9) {};
\draw (1,-9) -- (9,-1) {};
\draw[very thick] (-2,-9) -- (-7,-14) {};
\draw[very thick] (-2,-9) -- (3,-7) {};
\draw[very thick] (3,-7) -- (9,-1) {};

\draw (-2,-8.8)--(-2,-9.2) node[below] {$a$};
\draw (1,-8.8)--(1,-9.2) node[below] {$b$};

\node at (4.8,-4) {$\varphi + \psi$};
\node at (2,-8.5) {$f$};

\draw[dotted] (-2,-9) -- (-2,3.5) node[below left] {$(a,u_\mu(a))$};
\node[circle,very thick,draw=black,fill=white,minimum size = 4pt,inner sep=0pt] at (-2,3.5) {};
\draw[dotted] (3,-7) -- (3,6) node[above left] {$(z,u_\nu(z))$};
\node[circle,very thick,draw=black,fill=white,minimum size = 4pt,inner sep=0pt] at (3,6) {};
\end{tikzpicture}} &

\resizebox{0.4\textwidth}{!}{
\begin{tikzpicture}[scale=0.5]
\draw[domain=-9:9,variable=\x, very thick] plot ({\x},{sqrt(24+\x*\x)});
\draw[domain=-9.5:9.5,variable=\x, very thick] plot ({\x},{sqrt(8.64+\x*\x)});
\draw[domain=-10:10,variable=\x, dashed] plot ({\x},{abs(\x)});

\draw[domain=-3:1,variable=\x] plot ({\x},{5 + 0.2*(\x-1)});

\draw[->,>=stealth] (-10,0) -> (10,0) {};

\node at (-5.3,7.8) {$u_\nu$};
\node at (-5.8,7) {$u_\mu$};

\draw[->,>=stealth] (-10,-9) -> (10,-9) {};
\draw[very thick] (2,-9) -- (9,-2) node[left] {$\varphi + \psi$};
\draw[very thick] (-3,-9) -- (-8,-14) {};
\draw[very thick] (-3,-9) -- (2,-9) {};

\draw (-3,-8.8)--(-3,-9.2) node[below] {$a$};
\draw (2,-8.8)--(2,-9.2) node[below] {$b$};

\draw[dotted] (-3,-9) -- (-3,4.2) node[below left] {$(a,u_\mu(a))$};
\node[circle,very thick,draw=black,fill=white,minimum size = 4pt,inner sep=0pt] at (-3,4.2) {};
\draw[dotted] (1,-9) -- (1,5) node[above left] {$(z,u_\nu(z))$};
\node[circle,very thick,draw=black,fill=white,minimum size = 4pt,inner sep=0pt] at (1,5) {};
\end{tikzpicture}}
\end{tabular}
\caption{Construction of the potential functions of the optimal intermediate laws $\theta$ (top) and the dual optimizers $\varphi + \psi$ (bottom) for a risk reversal as described in Proposition~\ref{prop:risk reversal}; $z > b$ in the left panel and $z < b$ in the right panel.}
\label{fig:risk reversal}
\end{figure}
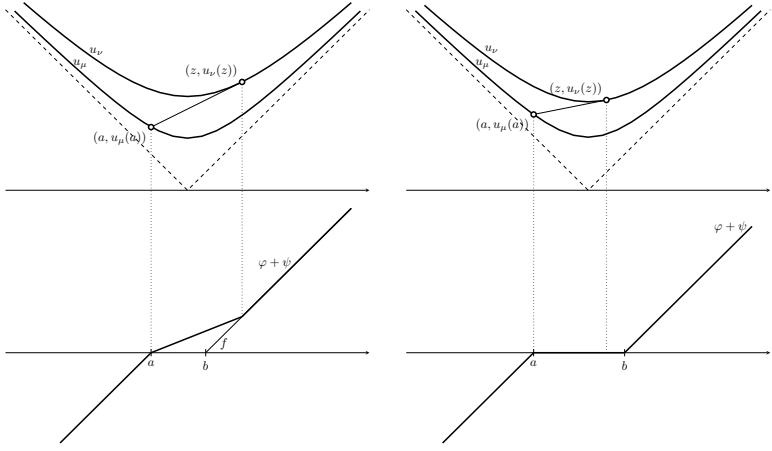

\begin{proposition}
\label{prop:risk reversal}
Consider the line through the point $(a,u_\mu(a))$ of maximal slope lying below (or on) the graph of $u_\nu$; cf.~Figure~\ref{fig:risk reversal}. This line is either (i) a tangent line to the graph of $u_\nu$ with a tangent point $(z,u_\nu(z))$ for some $z \in (a,\infty)$ or (ii) the asymptote line for the graph of $u_\nu$ near $+\infty$.\footnote{Note that case (ii) can only happen when $(a,\mu(a))$ lies on the increasing part of the dashed potential function in Figure~\ref{fig:risk reversal}. In particular, in this case, $\mu$ is concentrated on the left of $a$.}

In case (i), define the concave function $\varphi$ and the convex function $\psi$ by
\begin{align*}
\varphi(x)
&=
 - \alpha (x-a)_+,\\
\psi(x)
&= x-a + \alpha(x-(z\vee b))_+,
\end{align*}
where $\alpha = (b-a)/((z\vee b)-a)$. Moreover, let $u$ be the unique convex function that coincides with $u_\mu$ on $(-\infty,a]$ and with $u_\nu$ on $[z,\infty)$ and is affine on $[a,z]$ (i.e., $u$ coincides on $[a,z]$ with the tangent line considered above). Denote by $\theta$ the unique measure with potential function $u_\theta = u$. In case (ii), set $\varphi(x) = 0$, $\psi(x) = x-a$, and $\theta = \mu$.

Then, $\theta$ is an optimizer for the auxiliary primal problem $\wt\bfS_{\mu,\nu}(f)$, $(\varphi,\psi)$ is an optimizer for the auxiliary dual problem $\wt \bfI_{\mu,\nu}(f)$, and the common optimal optimal value is given in terms of the potential functions of $\mu$ and $\nu$ by
\begin{align*}
\wt\bfS_{\mu,\nu}(f)
= \wt\bfI_{\mu,\nu}(f)
= 
\begin{cases}
m_1 - \frac{a+b}{2}m_0 + \frac{b-a}{2}\frac{u_\nu(z\vee b) - u_\mu(a)}{(z\vee b)-a} & \text{in case (i)},\\
m_1 - a m_0 & \text{in case (ii)}.
\end{cases}
\end{align*}
\end{proposition}

\begin{proof}
We first note that $\theta$ and $(\varphi, \psi)$ are admissible elements for the auxiliary primal and dual problems, respectively. Indeed, by construction, $u_\theta$ is convex and lies between $u_\mu$ and $u_\theta$. Thus, the associated measure $\theta$ satisfies $\mu \leq_c \theta \leq_c \nu$. Moreover, a straightforward computation shows that $\varphi + \psi \geq f$, and $(\varphi,\psi) \in L^c(\mu,\nu)$ by Lemma~\ref{lem:ConvexIntegrationRules}.

By the weak duality inequality \eqref{eqn:lem:ContinuousDuality:pf:weak duality} (this also holds if $f$ is bounded from below by an affine function; cf.~Remark~\ref{rem:auxiliary duality:relax lower bound}), $\theta(f) \leq \mu(\varphi) + \nu(\psi)$ holds for any admissible primal and dual elements. It thus suffices to show that $\theta(f) = \mu(\varphi) + \nu(\psi)$ for our particular choices for $\theta$ and $(\varphi,\psi)$. 

\emph{Case (i):} Using the identity $(t-s)_+ = \frac{1}{2}(\vert t-s \vert + t-s)$, the integrals $\theta(f), \mu(\varphi)$, and $\nu(\psi)$ can be expressed in terms of the potential functions of $\mu,\theta$, and $\nu$ as follows:
\begin{align*}
\theta(f)
&= \frac{1}{2}(u_\theta(b) - u_\theta(a)) + m_1 - \frac{a+b}{2}m_0,\\
\mu(\varphi)
&= -\frac{\alpha}{2}(u_\mu(a) + m_1 - am_0),\\
\nu(\psi)
&= m_1 - am_0 + \frac{\alpha}{2}(u_\nu(z \vee b) + m_1 - (z\vee b)m_0).
\end{align*}
Substituting $\alpha = (b-a)/((z\vee b)-a)$ and simplifying gives
\begin{align*}
\mu(\varphi) + \nu(\psi)
&= m_1 + \frac{1}{2}((a - (z\vee b)) \alpha - 2a) m_0 - \frac{\alpha}{2}(u_\mu(a) - u_\nu(z\vee b))\\
&= m_1 - \frac{a+b}{2}m_0 + \frac{b-a}{2}\frac{u_\nu(z\vee b) - u_\mu(a)}{(z\vee b)-a}.
\end{align*}
Hence,
\begin{align*}
\theta(f) - (\mu(\varphi) + \nu(\psi))
&= \frac{b-a}{2}\left(\frac{u_\theta(b) - u_\theta(a)}{b-a} - \frac{u_\nu(z\vee b) - u_\mu(a)}{(z\vee b)-a}\right),
\end{align*}
and it suffices to show that the two quotients inside the brackets are equal. To this end, we distinguish two cases. On the one hand, if $z \leq b$, it is enough to observe that $u_\theta(a) = u_\mu(a)$ and $u_\theta(b) = u_\nu(b)$ by the construction of $u = u_\theta$. On the other hand, if $z \geq b$, then the two quotients are the same because $u = u_\theta$ is affine on $[a, z] \supset [a,b]$ and coincides with $u_\mu$ at $a$ and with $u_\nu$ at $z$.

\emph{Case (ii):} On the one hand, since $\theta = \mu$ and $\mu$ is concentrated on the left of $a$, we have $\theta(f) = \mu(f)
= m_1-am_0$. On the other hand, $\mu(\varphi) + \nu(\psi) = \int (x-a) \,\nu(\diff x) = m_1 - am_0$.
\end{proof}

\subsection{Butterfly spreads}

The payoff function of a butterfly spread is of the form
\begin{align*}
f(x)
&= (x-(a-h))_+ - 2(x-a)_+ + (x-(a+h))_+, \\
\end{align*}
for fixed $a$ and $h > 0$. We have the following analog to Proposition~\ref{prop:risk reversal}; we omit the proof.

\begin{figure}
\centering
\resizebox{0.4\textwidth}{!}{
\begin{tikzpicture}[scale=0.5]
\draw[domain=-9:9,variable=\x, very thick] plot ({\x},{sqrt(24+\x*\x)});
\draw[domain=-9.5:9.5,variable=\x, very thick] plot ({\x},{sqrt(9+\x*\x)});
\draw[domain=-10:10,variable=\x, dashed] plot ({\x},{abs(\x)});
\draw[domain=0:6.3245,variable=\x] plot ({\x},{3 + 0.790569*\x});
\draw[domain=-6.3245:0,variable=\x] plot ({\x},{3 - 0.790569*\x});

\draw[->,>=stealth] (-10,0) -> (10,0) {};

\node at (-4.6,7.2) {$u_\nu$};
\node at (-5.2,6.5) {$u_\mu$};

\draw[->,>=stealth] (-10,-5) -> (10,-5) {};
\draw (-3,-5) -- (0,-2) {};
\draw (0,-2) -- (3,-5) {};
\draw[very thick] (-6.3245,-5) -- (0,-2) {};
\draw[very thick] (6.3245,-5) -- (0,-2) {};
\draw[very thick] (-10,-5) -- (-6.3245,-5) {};
\draw[very thick] (6.3245,-5) -- (10,-5) node[above left] {$\varphi + \psi$};

\draw (-3,-4.8)--(-3,-5.2) node[below] {$a-h$};
\draw (0,-4.8)--(0,-5.2) node[below] {$a$};
\draw (3,-4.8)--(3,-5.2) node[below] {$a+h$};

\node at (1.2,-3.8) {$f$};

\draw[dotted] (-6.3245,-5) -- (-6.3245,8) node[above right] {$(z_-,u_\nu(z_-))$};
\node[circle,very thick,draw=black,fill=white,minimum size = 4pt,inner sep=0pt] at (-6.3245,8) {};
\draw[dotted] (6.3245,-5) -- (6.3245,8) node[above left] {$(z_+,u_\nu(z_+))$};
\node[circle,very thick,draw=black,fill=white,minimum size = 4pt,inner sep=0pt] at (6.3245,8) {};
\draw[dotted] (0,-2) -- (0,3) node[below left] {$(a,u_\mu(a))$};
\node[circle,very thick,draw=black,fill=white,minimum size = 4pt,inner sep=0pt] at (0,3) {};
\end{tikzpicture}}
\caption{Construction of the potential function of the optimal intermediate law $\theta$ (top) and the dual optimizer $\varphi + \psi$ (bottom) for a butterfly spread as described in Proposition~\ref{prop:butterfly}.}
\label{fig:butterfly}
\end{figure}
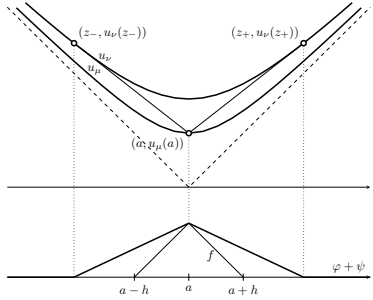

\begin{proposition}
\label{prop:butterfly}
Consider the two lines $l_+,l_-$ through the point $(a,u_\mu(a))$ of maximal and minimal slope, respectively, lying below (or on) the graph of $u_\nu$. We distinguish the cases (i$+$) $l_+$ is a tangent line with tangent point $(z_+,u_\nu(z_+))$,
(ii$+$) $l_+$ is an asymptote, (i$-$) $l_-$ is a tangent line with tangent point $(z_-,u_\nu(z_-))$, (ii$-$) $l_-$ is 
an asymptote. In case (ii$\pm$), we set $z_\pm = \pm \infty$.

Let $u$ be the convex function that coincides with $u_\nu$ on $(-\infty,z_-] \cup [z_+,\infty)$ and is affine on $[z_-,a]$ and on $[a,z_+]$, and define the concave function $\varphi$ and the convex function $\psi$ by
\begin{align*}
\varphi(x)
&= -(\alpha+\beta)(x-a)_+,\\
\psi(x)
&= \alpha(x-(z_- \wedge (a-h)))_+ + \beta(x-(z_+ \vee (a+h)))_+,
\end{align*}
where $\alpha = \frac{h}{a-(z_- \wedge (a-h))}$ and $\beta = \frac{h}{(z_+ \vee (a+h)) - a}$. Here, in the asymptote cases (ii$\pm$), $\varphi,\psi$ need to be interpreted as the limiting functions that arise as $z_\pm \to \pm\infty$.\footnote{For instance, if $z_- = -\infty$ and $z_+ < \infty$, then $\varphi(x) = -\beta(x-a)_+$ and $\psi(x) = h + \beta(x-(z_+ \vee (a+h)))_+$.}

Then, the intermediate law $\theta$ with potential function $u_\theta = u$ is an optimizer for the auxiliary primal problem $\wt\bfS_{\mu,\nu}(f)$, $(\varphi,\psi)$ is an optimizer for the auxiliary dual problem $\wt \bfI_{\mu,\nu}(f)$, and the common optimal optimal value is given in terms of the potential functions of $\mu$ and $\nu$ by
\begin{align*}
\wt\bfS_{\mu,\nu}(f)
= \wt\bfI_{\mu,\nu}(f)
= \tfrac{h}{2} (s_+ + s_-),
\end{align*}
where
\begin{align*}
s_+
&=
\begin{cases}
\frac{u_\nu(z_+ \vee (a+h)) - u_\mu(a)}{(z_+ \vee (a+h)) - a} &\text{in case (i+)},\\
m_0 & \text{in case (ii+)},
\end{cases}
\\
s_-
&=
\begin{cases}
\frac{u_\nu(z_- \wedge (a-h))-u_\mu(a)}{a-(z_- \wedge (a-h))}&\text{in case (i-)},\\
-m_0 & \text{in case (ii-)}.
\end{cases}
\end{align*}
\end{proposition}

\section{Counterexamples}
\label{sec:counterexamples}

In this section, we give four counterexamples. Example~\ref{ex:globally convex concave} shows that strong duality for the auxiliary problems may fail for general (not necessarily irreducible) marginals if the dual elements $\varphi,\psi$ are required to be \emph{globally} concave and convex, respectively. Example~\ref{ex:NoRegularIntegrability} shows that strong duality may fail if the dual elements $\varphi$ and $\psi$ are required to be $\mu$- and $\nu$-integrable, respectively.
Example~\ref{ex:multi marginal} shows that the robust model-based prices of Asian- and American-style derivatives are typically not equivalent when more than two marginals are given. Example~\ref{ex:dual auxiliary fail} shows that the equality $\bfS_{\mu,\nu}(f,\cA)= \wt\bfS_{\mu,\nu}(f)$ may fail when the assumptions of Proposition~\ref{prop:reduction:primal} are violated.

\begin{example}[Duality gap with globally convex/concave dual elements]
\label{ex:globally convex concave}
Let $\mu = \frac{1}{2}\delta_{-1} + \frac{1}{2}\delta_1$, let $\nu$ be the uniform distribution on $(-2,2)$, and set $f(x) := \vert x \vert^{-\frac{1}{2}}$, $x \in \RR$ (with $f(0) = \infty$).

First, we show that $\wt\bfS_{\mu,\nu}(f)$ is finite. Fix any $\mu\leq_c\theta\leq_c\nu$. Computing the potential functions $u_\mu$ and $u_\nu$ shows that $\mu \leq_c \nu$ and that $\lbrace u_\mu < u_\nu \rbrace = I_1 \cup I_2$ with $I_1 = (-2,0)$ and $I_2 = (0,2)$. Because $\nu$ does not have an atom at the common boundary $0$ of $I_1$ and $I_2$, also $\theta$ cannot have an atom at $0$. Thus, we can write $\theta = \theta_1 + \theta_2$ with
\begin{equation*}
\frac{1}{2}\delta_{-1}
\leq_c \theta_1
\leq_c \nu|_{I_1}
\quad\text{and}\quad
\frac{1}{2}\delta_{1}
\leq_c \theta_2
\leq_c \nu|_{I_2}.
\end{equation*}
Since $f$ is convex when restricted to $I_1$ or $I_2$, we have
\begin{equation*}
\theta(f)
= \theta_1(f) + \theta_2(f)
\leq \nu|_{I_1}(f) + \nu|_{I_2}(f)
= \nu(f)
< \infty.
\end{equation*}
It follows $\wt\bfS_{\mu,\nu}(f) = \nu(f) < \infty$.

Second, let $\varphi$ be concave and $\psi$ be convex such that $\varphi + \psi \geq f$. We show that then necessarily $\mu(\varphi) + \nu(\psi) = \infty$. To this end, we may assume that $\varphi < \infty$ on $\supp(\mu) = \lbrace -1, 1\rbrace$. Then $\varphi < \infty$ everywhere by concavity. Thus, evaluating $\varphi + \psi \geq f$ at $0$ implies that $\psi(0) = \infty$. Therefore, $\psi = \infty$ on $(-\infty,0]$ or on $[0,\infty)$ by the convexity of $\psi$. In both cases, we have $\mu(\varphi) + \nu(\psi) = \infty$.
\end{example}

\begin{example}[Duality gap with individually integrable dual elements]
\label{ex:NoRegularIntegrability}
We consider the marginals
\begin{equation*}
\mu
:= C \sum_{n\geq 1} n^{-3}\mu_n
\quad\text{and}\quad
\nu
:= C\sum_{n\geq 1} n^{-3} \nu_n,
\end{equation*}
where $C := (\sum_{n\geq 1} n^{-3})^{-1}$, $
\mu_n
:= \delta_n$ and $\nu_n
:= \frac{1}{3}(\delta_{n-1}+\delta_n + \delta_{n+1})$
 for $n\geq 1.$
These are the same marginals as in \cite[Example~8.5]{BeiglbockNutzTouzi2017} where it is shown that $\mu\leq_c\nu$ is irreducible with domain $((0,\infty),[0,\infty))$. We now let $f:\RR_+\to[0,1]$ be the piecewise linear function through the points given by $f(n) = 0$ and $f(2n + \tfrac{1}{2}) = \tfrac{1}{4}$ for $n \geq 0$; cf.~Figure~\ref{fig:ex:NoRegularIntegrability}.

We proceed to construct candidates for optimizers for $\wt\bfS_{\mu,\nu}(f)$ and $\wt\bfI_{\mu,\nu}(f)$. For the primal problem, define the sequence $(\bar\theta_n)_{n \geq 1}$ by
\begin{align*}
\bar\theta_n
&=
\begin{cases}
\frac{1}{3}(\delta_{n-1} + 2\delta_{n+\frac{1}{2}}) & \text{for } n \text{ even}, \\
\frac{1}{3}(2\delta_{n-\frac{1}{2}} + \delta_{n+1}) & \text{for } n \text{ odd},
\end{cases}  
\end{align*}
and set $\bar\theta := C\sum_{n\geq 1} n^{-3} \bar\theta_n$. One can check that $\mu_n \leq_c \bar\theta_n \leq_c \nu_n$ and compute $\bar\theta_n(f) = \frac{1}{6}$. Hence, $\mu \leq_c \bar\theta \leq_c \nu$ (by linearity of potential functions in the measure) and $\bar\theta(f) = \frac{1}{6}$.

\begin{figure}
\begin{center}
\begin{tikzpicture}
\draw[->] (-0.05,0) node[left] {$0$} -- (5,0) node[right] {$\RR_+$};
\draw[->] (0,-0.05) -- (0,1) {};
\draw[very thick] (0,0) -- (0.5,0.5) -- (1,0) -- (2,0) -- (2.5,0.5) -- (3,0) -- (4,0) -- (4.5,0.5) {};
\draw (0.5,0.05) -- (0.5,-0.05) node[below] {$\frac{1}{2}$};
\draw (1,0.05) -- (1,-0.05) node[below] {$1$};
\draw (1.5,0.05) -- (1.5,-0.05) node[below] {};
\draw (2,0.05) -- (2,-0.05) node[below] {$2$};
\draw (2.5,0.05) -- (2.5,-0.05) node[below] {$\frac{5}{2}$};
\draw (3,0.05) -- (3,-0.05) node[below] {$3$};
\draw (3.5,0.05) -- (3.5,-0.05) node[below] {};
\draw (4,0.05) -- (4,-0.05) node[below] {};
\draw (4.5,0.05) -- (4.5,-0.05) node[below] {};
\draw (0.05,0.5) -- (-0.05,0.5) node[left] {$\frac{1}{4}$};
\end{tikzpicture}
\end{center}
\caption{The function $f$ in Example \ref{ex:NoRegularIntegrability}.}
\label{fig:ex:NoRegularIntegrability}
\end{figure}
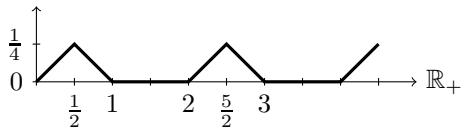

We now turn to the dual problem. Let $\bar\varphi$ and $\bar\psi$ be the unique concave and convex functions, respectively, with second derivative measures
\begin{equation*}
-\bar\varphi''
= \sum_{n \geq 0} \delta_{2n+\frac{1}{2}}
\quad\text{and}\quad
\bar\psi''
= \frac{1}{2}\sum_{n \geq 1}\delta_n
\end{equation*}
and $\bar\varphi(0) = \bar\psi(0) = 0$, $\bar\varphi'(0) = f'(0) = \frac{1}{2}$, and $\bar\psi'(0) = 0$. The ``initial conditions'' are chosen such that $f(0) = \bar\varphi(0) + \bar\psi(0)$ and $f'(0) = \bar\varphi'(0) + \bar\psi'(0)$ and the choice of the second derivative measures ensures that $\varphi$ and $\psi$ pick up the negative and positive curvature of $f$, respectively. Thus, $\bar\varphi + \bar\psi = f$ on $\RR_+$ by construction. We proceed to compute $\mu(\bar\varphi) + \nu(\bar\psi)$ in the sense of Definition~\ref{def:concave integral}. (The individual integrals are infinite because $\bar\varphi$ and $\bar\psi$ have quadratic growth while $\mu$ and $\nu$ have no second moments.) To this end, we note that $\bar\varphi + \bar\psi = f$ vanishes on the support of $\nu$. This implies that $\bar\varphi$ is a concave moderator for $(\bar\varphi,\bar\psi)$ with respect to $\mu\leq_c\nu$. We can then compute
\begin{align*}
\mu(\bar\varphi) + \nu(\bar\psi)
&= \mu(\bar\varphi-\bar\varphi) + \nu(\bar\psi + \bar\varphi) +  (\mu-\nu)(\bar\varphi)
= (\mu-\nu)(\bar\varphi) \\
&= C \sum_{n \geq 1} n^{-3} (\mu_n - \nu_n)(\bar\varphi).
\end{align*}
Fix $n \geq 1$. Because $\bar\varphi$ is continuous, we have
\begin{equation}
\label{eqn:ex:NoRegularIntegrability:concave integral}
(\mu_n - \nu_n)(\bar\varphi)
= \frac{1}{2}\int_I (u_{\mu_n} - u_{\nu_n}) \dd\bar\varphi''.
\end{equation}
The difference $u_{\mu_n} - u_{\nu_n}$ vanishes outside $(n-1,n+1)$ and on this interval, $\bar\varphi''$ is concentrated on either $n-\frac{1}{2}$ (if $n$ is odd) or on $n+\frac{1}{2}$ (if $n$ is even) with mass $1$. Therefore, the right-hand side of \eqref{eqn:ex:NoRegularIntegrability:concave integral} collapses to $\frac{1}{2}(u_{\mu_n} - u_{\nu_n})(n\pm\frac{1}{2}) = \frac{1}{6}$. It follows that $\mu(\bar\varphi) + \nu(\bar\psi) = \frac{1}{6} = \bar\theta(f)$. Hence, by (weak) duality, $\bar\theta$ and $(\bar\varphi,\bar\psi)$ are primal and dual optimizers, respectively.

We are now in a position to argue that no dual optimizer lies in $L^1(\mu) \times L^1(\nu)$. Suppose for the sake of contradiction that $(\varphi,\psi) \in L^1(\mu)\times L^1(\nu)$ is a dual optimizer and note that $\supp(\bar\theta) = \lbrace 0.5, 1, 2, 2.5, 3,\ldots \rbrace$. We have $\varphi + \psi = f$ $\bar\theta$-a.e.~by Proposition~\ref{prop:DualProperties}~(i). One can show that the following modifications of $(\varphi,\psi)$ do not affect its optimality nor the individual integrability of $\varphi$ and $\psi$; we omit the tedious details. First, $\psi$ is replaced by its piecewise linear interpolation at the atoms of $\nu$. Second, $\varphi$ is replaced by its piecewise linear interpolation at the kinks of $f$. Third, a suitable convex function is added to $\varphi$ and subtracted from $\psi$ (preserving their concavity and convexity, respectively) such that the second derivative measures $-\varphi''$ and $\psi''$ become singular.

Because $\varphi + \psi = f$ on $\supp(\bar\theta)$ and both sides are piecewise linear, we conclude that $\varphi + \psi = f$ holds on $[\frac{1}{2},\infty)$. As $-\varphi''$ and $\psi''$ are singular, $\varphi$ and $\psi$ must then account for the negative and positive curvature of $f$, respectively. It follows that both $\varphi$ and $\psi$ have quadratic growth. Since $\mu$ and $\nu$ do not have a second moment, we conclude that $\mu(\varphi) = -\infty$ and $\nu(\psi) = \infty$, a contradiction.
\end{example}

\begin{example}[Different robust model-based prices for Asian- and American-style derivatives for multiple marginals]
\label{ex:multi marginal}
For $n \geq 2$ given marginals $\mu_0\leq_c \mu_1 \leq_c \cdots \leq_c\mu_n$ corresponding to the time points $0,1,\ldots,n$ (say), the robust model-based price $\bfS_{\mu_0,\dots,\mu_n}(f,\cA)$ is defined analogously. But this robust model-based price now depends non-trivially on $\cA$, as the following example shows. Fix a strictly convex function $f$. On the one hand, if $\cA$ corresponds to American-style derivatives, then one can check that $\bfS_{\mu_0,\dots,\mu_n}(f,\cA) = \mu_n(f)$. On the other hand, for Asian-style derivatives, i.e., $\cA' = \{t \mapsto t/n\}$, Jensen's inequality yields
\begin{align*}
f\left(\frac{1}{n}\int_0^n X_t \dd t\right)
\leq \frac{1}{n} \sum_{i=0}^{n-1} f\left(\int_i^{i+1} X_t \dd t \right), 
\end{align*}
so that
\begin{align*}
\bfS_{\mu_0,\dots,\mu_n}(f,\cA')
\leq \frac{1}{n}\sum_{i=1}^n \mu_i(f) \leq \mu_n(f).
\end{align*}
For a generic choice of marginals, both inequalities are strict. Hence, the robust model-based price of an Asian-style derivative with a strictly convex payoff function is typically smaller than that of the corresponding American-style derivative.
\end{example}

\begin{example}[Necessity of the assumptions of Proposition~\ref{prop:reduction:primal}]\quad
\label{ex:dual auxiliary fail}
\begin{enumerate}
\item We show that $\bfS_{\mu,\nu}(f,\cA) = \wt\bfS_{\mu,\nu}(f)$ may fail if $\cA$ does not contain an interior averaging process. Set $\cA = \lbrace A \rbrace = \{t \mapsto \tfrac{1}{2} + \tfrac{1}{2}\1_{\lbrace t = T \rbrace}\}$, so that $\int_0^T X_t dA_t = (X_0 + X_T)/2$, and consider $f(x) = x^2$. Then, using the martingale property of $X$ under any $P \in \cM(\mu,\nu)$, one can check that $\bfS_{\mu,\nu}(f,\cA) = (3\mu(f) + \nu(f))/4$, whereas
$\wt\bfS_{\mu,\nu}(f) = \nu(f)$ since $f$ is convex. Now, choose $\mu$ and $\nu$ such that $\mu(f) < \nu(f)$ ($f$ is strictly convex). Then, $\bfS_{\mu,\nu}(f,\cA) < \wt\bfS_{\mu,\nu}(f)$.

\item We show that $\bfS_{\mu,\nu}(f,\cA) = \wt\bfS_{\mu,\nu}(f)$  may fail if $\cA$ contains an interior averaging process but $f$ is not lower semicontinuous. Set $\cA = \{t \mapsto t/T\}$ and $f(x) = \1_{\lbrace \vert x \vert \geq 1 \rbrace}$, and choose $\mu = \delta_0$ and $\nu = (\delta_1 + \delta_{-1})/2$. On the one hand, since $\nu(f) = 1$ and $f \leq 1$, we have $\wt\bfS_{\mu,\nu}(f) = 1$. On the other hand, we claim that $\bfS_{\mu,\nu}(f,\cA) = 0$. To this end, fix $P \in \cM(\mu,\nu)$. Since $P$-a.e.~path of $X$ starts in $0$, evolves in $[-1,1]$, and is right-continuous, $\left\vert \frac{1}{T}\int_0^T X_t \dd t \right\vert < 1$ $P$-a.s. Thus, $E^P\big[f(\frac{1}{T}\int_0^T X_t \dd t)\big] = 0$. Since $P \in \cM(\mu,\nu)$ was arbitrary, $\bfS_{\mu,\nu}(f,\cA) = 0$.
\end{enumerate}
\end{example}

\small
\bibliographystyle{amsplain}
\bibliography{bibliography}
\end{document}